\documentclass[twoside]{amsart}
\usepackage{amsfonts, amsmath, amssymb}
\usepackage[linktocpage=true, colorlinks=true, linkcolor=magenta, filecolor=magenta,urlcolor=cyan,citecolor=cyan]{hyperref}
\usepackage{graphicx}
\usepackage{aligned-overset}
\usepackage{float}
\usepackage{srcltx}
\usepackage[all]{xy}
\usepackage{bbm}
\usepackage{CJKutf8}
\usepackage[utf8]{inputenc}	
\usepackage[T2A]{fontenc}
\usepackage{xcolor}
\usepackage[normalem]{ulem}
\usepackage{bm}
\usepackage{pifont}
\usepackage{latexsym}
\usepackage{stmaryrd}
\usepackage[2emode]{psfrag}
\usepackage{yhmath}
\usepackage{array}
\usepackage{dsfont}
\usepackage{mathrsfs}
\usepackage{tikz-cd}
\usepackage{comment}
\usepackage{subcaption}
\usetikzlibrary { decorations.pathmorphing, decorations.pathreplacing, decorations.shapes }
\usetikzlibrary{graphs}

\definecolor{Maroon}{rgb}{0.6 0 0}
\definecolor{Prussian}{rgb}{0.05 0 0.6}
\definecolor{Emerald}{rgb}{0 0.5 0.1}

\newtheoremstyle{mytheorem}%
{10.0pt plus 2.0pt minus 2.0pt} 
{10.0pt plus 2.0pt minus 2.0pt} 
{\itshape} 
{} 
{\sc} 
{.} 
{ } 
{} 

\newtheoremstyle{mydefinition}%
{10.0pt plus 2.0pt minus 2.0pt} 
{10.0pt plus 2.0pt minus 2.0pt} 
{} 
{} 
{\sc} 
{.} 
{ } 
{} 

\newtheoremstyle{myexample}%
{10.0pt plus 2.0pt minus 2.0pt} 
{10.0pt plus 2.0pt minus 2.0pt} 
{\small} 
{} 
{\sc} 
{.} 
{ } 
{} 

\newtheoremstyle{myremark}%
{10.0pt plus 2.0pt minus 2.0pt} 
{10.0pt plus 2.0pt minus 2.0pt} 
{} 
{} 
{\itshape} 
{.} 
{ } 
{} 

\theoremstyle{mytheorem}
\newtheorem{theorem}{Theorem}[section]
\newtheorem{lemma}[theorem]{Lemma}

\newtheorem{corollary}[theorem]{Corollary}
\newtheorem{proposition}[theorem]{Proposition} 

\theoremstyle{myremark}
\newtheorem{remark}[theorem]{Remark}
\newtheorem{convention}[theorem]{Convention}
\newtheorem{notation}[theorem]{Notation}

\theoremstyle{mydefinition}
\newtheorem{definition}[theorem]{Definition}

\theoremstyle{myexample}
\newtheorem{example}[theorem]{Example}
\newtheorem{counterexample}[theorem]{Counterexample}

\newtheoremstyle{myzusatz}
 {10.0pt plus 2.0pt minus 2.0pt} 
{10.0pt plus 2.0pt minus 2.0pt} 
{\itshape} 
{} 
{\sc} 
{.} 
{ } 
{\thmname{#1}\thmnumber{ #2}\thmnote{ #3}}

\theoremstyle{myzusatz}
\newtheorem{zusatz}[theorem]{Addendum}

\definecolor{gray1}{gray}{0.8}
\definecolor{gray2}{gray}{0.6}
\definecolor{gray3}{gray}{0.4}
\definecolor{gray4}{gray}{0.2}

\allowdisplaybreaks
\newcommand{\id}{\mathrm{id}}

\newcommand{\Cc}{\mathscr{C}}

\newcommand{\Gg}{\mathscr{G}}

\newcommand{\Ss}{\mathscr{S}}

\newcommand{\Hh}{\mathscr{H}}
\newcommand{\Kk}{\mathscr{K}}
\newcommand{\Mu}{\mathrm{M}}

\newcommand{\inv}[1]{{#1}^{-1}}
\newcommand{\ot}{\otimes}

\newcommand{\set}[1]{\lbrace #1\rbrace}

\def\Quiv{{\mathcal{Q}}}
\def\source{\mathfrak{s}}
\def\target{\mathfrak{t}}

\DeclareMathOperator{\Hol}{\mathrm{Hol}}
\DeclareMathOperator{\Aut}{\mathrm{Aut}}
\DeclareMathOperator{\Obj}{\mathrm{Obj}}
\DeclareMathOperator{\mmaps}{\begin{smallmatrix}{\mapsto}\\ {\mapsfrom}\end{smallmatrix}}

\newcommand{\Z}{\mathbb{Z}}

\newcommand{\blank}{\raisebox{-2pt}{\text{---}}}

{
\left\{\begin{aligned}}
{\end{aligned}
\right.
}

\newcommand{\One}{\mathbbm{1}}
\DeclareMathOperator{\im}{\mathrm{im}}

\begin{document}
\hyphenation{qua-si-grou-po-id}
\hyphenation{qua-si-grou-po-ids}
\hyphenation{grou-po-id}
\hyphenation{grou-po-ids}
\hyphenation{oid-i-fi-ca-tion}
\hyphenation{a-quo-id}
\hyphenation{a-quo-ids}

\def\presub#1#2#3%
  {\mathop{}%
   \mathopen{\vphantom{#2}}_{#1}%
   \kern-\scriptspace%
   {#2}_{#3}}

\newcommand{\fibre}[2]{\presub{#1}{\times}{#2}}

\newcommand{\twomapsright}[2]{\,\underset{#2}{\overset{#1}{\rightrightarrows}} \, }

\title{On dynamical skew braces and skew bracoids}
\author{Davide Ferri}
\begin{abstract}\noindent Dynamical skew braces are known to produce solutions to the quiver-theoretic Yang--Baxter equation. Under a technical hypothesis, we prove that these solutions are braided groupoids (and hence skew bracoids in the sense of Sheng, Tang, and Zhu). Conversely, every connected braided groupoid can be \emph{parallelised}, making it isomorphic to a dynamical skew brace. We study the combinatorics of these objects, depending on some strings of integer invariants.
\end{abstract}
\address{%
\parbox[b]{0.9\linewidth}{University of Turin, Department of Mathematics ``G.\@ Peano'',\\ via
 Carlo Alberto 10, 10123 Torino, Italy.\\
 Vrije Universiteit Brussel, Department of Mathematics and Data Science,\\ Pleinlaan 2, 1050, Brussels, Belgium.}}
 \email{d.ferri@unito.it, Davide.Ferri@vub.be}
\maketitle
\tableofcontents 

\section{Introduction}\noindent The Yang--Baxter equation (\textsc{ybe}) arised in the theory of quantum integrable systems \cite{yang} and statistical mechanics \cite{BAXTER1973}, and is now studied in both physics and mathematics, being transversal to the theories of quantum groups, low-dimension topology, groups, rings, Garside structures, and several other fields. 

A notion of \textsc{ybe} can be defined in every monoidal category $(\Cc,\ot,\One)$: a solution to the \textsc{ybe} in $\Cc$ is an object $X$ together with a morphism $\sigma\colon X^{\ot 2}\to X^{\ot 2}$ satisfying a braid relation. The \textsc{ybe} in the category of vector spaces $\mathsf{Vec}_\Bbbk$ is crucial in quantum integrability and statistical mechanics, and in the representation theory of braid groups. The \textsc{ybe} in $\mathsf{Set}$ has been proposed by Drinfeld \cite{drinfeld2006some} as a tool to produce solutions in $\mathsf{Vec}_\Bbbk$ by linearisation and deformation; but the topic was found interesting in its own regard, due to its overarching applications in mathematics, such as in radical rings, PI algebras, combinatorial knot theory, cycle sets, Garside groups---just as a small showcase.

Motivated by the theory of groupoids and of weak quasitriangular Hopf algebras \cite{ AguiarAndruskiewitsch, andruskiewitsch2005quiver}, the \textsc{ybe} in the category of quivers $\mathsf{Quiv}_\Lambda$ was proposed by Andruskiewitsch \cite{andruskiewitsch2005quiver} and Matsumoto and Shimizu \cite{matsumotoshimizu}. This is called the \emph{quiver-theoretic Yang--Baxter equation}, and constitutes the main topic of this paper. The quiver-theoretic \textsc{ybe} appears moreover in the theory of post-Lie algebroids \cite{sheng2024postgroupoidsquivertheoreticalsolutionsyangbaxter}, and has connections to geometry, Garside categories \cite{dehornoy2015foundations}, and the theory of heaps. 

Most notably, the quiver-theoretic \textsc{ybe} envelops, as a subcase, the set-theoretic \emph{dynamical Yang--Baxter equation} (\textsc{dybe}); see \cite{matsumotoshimizu}. The \textsc{dybe} is an equation arising in mathematical physics and conformal field theory (see \cite{felder1995conformal, gervais1984novel}), systematised by Etingof and Schiffmann \cite{etingof2001lectures}, and whose set-theoretic version was defined by Shibukawa \cite{shibu}.

It is natural to consider the theory of the set-theoretic \textsc{ybe}, and see what properties or definitions can be translated to quivers. This has been done in many places.

Solutions in $\mathsf{Set}$ have been studied for a long time, and the following picture emerged: \textit{i}) there are structures called \emph{braided groups}, whose explicit definition appears in Lu, Yan, and Zhu \cite{LYZ} after a categorical notion with the same name was given by Majid \cite{Majid_braidedgroups, MJbraided}; \textit{ii}) braided groups are equivalent to \emph{matched pairs of actions on groups} \cite{takeuchi1981matched}, which produce solutions to the set-theoretic \textsc{ybe}, see \cite{LYZ}; \textit{iii}) not all solutions come from braided groups, but every solution is ``enveloped'' by a braided group, whose underlying group is the \emph{structure group of the solution} defined by Etingof, Schedler, and Soloviev \cite{etingof1999set}; \textit{iv}) braided groups are equivalent to \emph{skew braces}, defined by Guarnieri and Vendramin \cite{guarnieri2017skew} generalising Rump \cite{rump2007braces}; and finally, \textit{v}) braided groups are also equivalent to \emph{post-groups}, introduced by Bai, Guo, Sheng, and Tang \cite{postgroups}, that are relevant objects in the theory of post-Lie algebras. 

The above picture has been reconstructed for the quiver-theoretic \textsc{ybe}: \textit{i}) \emph{braided groupoids} have been defined by Andruskiewitsch \cite{andruskiewitsch2005quiver}; \textit{ii}) they are equivalent to matched pairs of groupoids \cite{andruskiewitsch2005quiver,mackenzie}, and they provide indeed solutions to the quiver-theoretic Yang--Baxter equation; \textit{iii}) not all solutions are of this kind, but they are all ``enveloped'' by a braided groupoid, defined on their \emph{structure groupoid} \cite{andruskiewitsch2005quiver}; \textit{iv}) a quiver-theoretic notion of skew brace, which is equivalent to braided groupoids, has been recently introduced by Sheng, Tang, and Zhu \cite{sheng2024postgroupoidsquivertheoreticalsolutionsyangbaxter} with the name \emph{skew bracoid} (notice that the name was already used for a different structure, see \cite{martin2024skew}); and finally \textit{v}) Sheng, Tang, and Zhu introduced \emph{post-groupoids} \cite{sheng2024postgroupoidsquivertheoreticalsolutionsyangbaxter}, that are equivalent to braided groupoids, and are relevant in the theory of post-Lie algebroids. 

From the viewpoint of the \textsc{dybe}, a notion of \emph{dynamical skew brace} has been defined by Anai \cite{anai2021thesis} generalising Matsumoto \cite{matsu}. Dynamical skew braces produce solutions to the \textsc{dybe}---and hence to the quiver-theoretic \textsc{ybe}---, and they are very rich in combinatorics \cite{matsumoto2011combinatorialaspects}.

In \S\ref{sec:qtsb}, we give the definition of \emph{quiver-theoretic skew braces} (Definition \ref{def:qtsb}), which includes skew bracoids as a subcase. We prove in Addendum \ref{zusatz} that dynamical skew braces give rise to quiver-theoretic skew braces. In particular, \emph{zero-symmetric} dynamical skew braces produce skew bracoids. A notion of \textit{post-semiloopoid} is also defined, and we prove that quiver-theoretic skew braces are equivalent to post-semiloopoids.

In the remainder of \S\ref{sec:dbraces}, we investigate the combinatorics of dynamical skew braces. We observe that every finite group $A$ is associated with a canonical zero-symmetric dynamical skew brace (and thus with a canonical skew bracoid). The geometry of the associated quiver depends on a string of integer invariants $N^A_s$ (Remark \ref{maximal_symmetric_dbrace}), that remain hard to compute: $N^A_1$ in particular is the number of skew braces on $A$. We prove a relation for these invariants, and compute some small examples. 

Finally, in \S\ref{sec:aredbraces} we prove  our main result (Theorem \ref{thm:aredbraces}): all connected skew bracoids come from the connected zero-symmetric dynamical skew braces.

Theorem \ref{thm:aredbraces} is consistent with the fact that skew bracoids lie in the essential image of the Matsumoto--Shimizu functor $\Quiv$ (see Lemma \ref{essential_image_of_Q}). Moreover, it also has a geometric interpretation: to retrieve a dynamical skew brace from a skew bracoid, is to ``parallelise'' a discrete manifold in a way that is compatible with its ``fundamental groupoid'' multiplication.

\section{Preliminaries}
\noindent For a category $\Cc$, we denote by $\Obj(\Cc)$ its class of objects and, for $X,Y\in \Obj(\Cc)$, we denote by $\Cc(X,Y)$ the class of morphisms $X\to Y$. The composition of arrows $f\colon X\to Y$ and $g\colon Y\to Z$ is denoted by $g\circ f$, or simply by $gf$.

We shall indicate with $(\Cc, \ot ,\One)$ a monoidal category, with monoidal product $\ot$ and monoidal unit $\One$. We shall consistently neglect the associativity and unit constraints, and the MacLane Coherence Theorem ensures us that this negligence is mostly safe.
\subsection{Quivers} A quiver can be described in multiple equivalent ways: as a directed multigraph with loops; as a small precategory; or as the datum of arrows and vertices, and a \emph{source} and \emph{target} map. The latter will be our preferred approach.
\begin{definition}
    A \emph{quiver} $Q\twomapsright{\source}{\target}\Lambda$ is the datum of two nonempty sets $Q$ and $\Lambda$, and of two maps $\source,\target \colon Q\to \Lambda$. We call $Q$ the set of \emph{arrows}, $\Lambda$ the set of \emph{vertices}, and $\source$ and $\target$ the \emph{source} and \emph{target} map respectively.

    With an abuse of terminology, we shall often say that $Q$ is \emph{a quiver over $\Lambda$}. The source and target maps will be omitted when they are understood from the context, or subscripts such as $\source_Q,\target_Q$ will be added for clarity when needed.
\end{definition}
\begin{notation}
    We import, for quivers, the same notation used for categories: we denote by $Q(\lambda,\mu)$ the set of arrows with source $\lambda$ and target $\mu$, and by $Q(\lambda,\Lambda)$, resp.\@ $Q(\Lambda,\lambda)$, the set of all arrows with source, resp.\@ target, $\lambda$. If $Q$ is a quiver over $\Lambda$, we shall write $\Obj(Q) = \Lambda$.
\end{notation}
\begin{convention}\label{conv:no-isolated-vertices}
    We shall always assume that $\Lambda = \im(\source)\cup \im(\target)$.
\end{convention}
\begin{definition}
A \emph{bundle of loops} is a quiver of the form $Q\twomapsright{\source}{\source}\Lambda$. We shall use the terminology ``bundle of groups'', ``bundle of skew braces'', etc., for a bundle of loops $Q$ where each set $Q(\lambda, \lambda)$ is a group, a skew braces, etc. 

A \emph{loop bundle} for a quiver $Q$ over $\Lambda$ is a map $p\colon Q\to \Lambda$. The \emph{fibre} of $p$ at $\lambda\in\Lambda$ is the set $p^{-1}(\lambda)$. The datum of a loop bundle can always be read as a quiver with $\source = \target = p$, and in this case the quiver is a bundle of loops.

A \emph{group bundle} is a loop bundle $p\colon \Gg\to \Lambda$ such that the corresponding bundle of loops $\Gg\twomapsright{p}{p}\Lambda$ has the structure of a bundle of groups where all groups $\Gg_\lambda = \Gg(\lambda,\lambda)$ are pairwise isomorphic.
\end{definition}
If $Q$ is a quiver over $\Lambda$, a \emph{subquiver} of $Q$ is a quiver $R$ over $\Lambda'$, with $\Lambda'\subseteq \Lambda$ and $R\subseteq Q$. We say that $R$ is \emph{full} if $R(\lambda,\mu) = Q(\lambda,\mu)$ for all $\lambda,\mu\in \Lambda'$, and that $R$ is \emph{wide} if $\Lambda'= \Lambda$. 
\begin{definition}\label{def:morphisms_of_quivers}
    Let $Q$ be a quiver over $\Lambda$, and $R$ be a quiver over $\Mu$. A \emph{weak morphism of quivers} $f\colon Q\to R$ is a pair $f = (f^1, f^0)$, where $f^0\colon \Lambda\to \Mu$  is a set-theoretic map, and $f_ 1\colon Q\to R$ is a set-theoretic map satisfying
    \[ \source_R(f^1(x)) = f^0(\source_Q(x)),\quad  \target_R(f^1(x)) = f^0(\target_Q(x)),\]
    for all $x\in Q$.

    Now let $\Lambda = \Mu$. A \emph{morphism of quivers over $\Lambda$} is defined as a weak morphism $f = (f^1, f^0)\colon Q\to R$ such that $f^0 = \id_\Lambda$.

    The cateory of quivers over $\Lambda$ endowed with morphisms over $\Lambda$ is denoted by $\mathsf{Quiv}_\Lambda$.
\end{definition}
\begin{convention}
    For us, ``a morphism of quivers'' is short for ``a morphism over $\Lambda$''. If $f = (f^1, \id_\Lambda)$ is a morphism over $\Lambda$, we just identify $f$ with $f^1$.
\end{convention}
Let $Q,R$ be quivers over $\Lambda$, and $p\colon Q\to \Lambda$, $q\colon R\to \Lambda$ be two maps. We denote by $Q\fibre{p}{q} R$ the quiver over $\Lambda$ with arrows
\[ \set{(x,y)\in Q\times R\mid p(x) = q(y)}, \]
and with source and target maps $\source_{Q\fibre{p}{q} R}(x, y):= \source_Q(x)$, and $\target_{Q\fibre{p}{q} R}(x, y):= \target_R(y)$ respectively; see e.g.\@ \cite[\S1.1]{andruskiewitsch2005quiver}.

The category $\mathsf{Quiv}_\Lambda$ is monoidal (see \cite{matsumotoshimizu}), with monoidal product $Q\ot R:= Q\fibre{\target_Q}{\source_R}R$. A pair $(x,y)\in Q\ot R$ is said to be a pair of ``composable arrows'', and is denoted by $x\ot y$. The monoidal unit is denoted by $\One_\Lambda$, and is the bundle on $\Lambda$ with exactly one loop on each vertex. 
\begin{notation}
    The quiver of \emph{paths of length $n$} on $Q$ can be identified with the tensor power $Q^{\ot n}$. The path formed by a sequence of arrows $x_1,\dots, x_n$ will be denoted by $x_1\ot \dots \ot x_n$, accordingly.
\end{notation}
\begin{definition}
    A quiver $Q$ is called \emph{Schurian} if $|Q(\lambda,\mu)|\le 1$ for all vertices $\lambda,\mu$. We adapt this nomenclature from Van Oystaeyen and Zhang \cite{VanOystaeyenQuivers}.
\end{definition}
\begin{definition}\label{complete_quiver_of_degree_n}
Let $\Lambda$ be a nonempty set, $d$ any cardinal (possibly infinite). A \emph{complete quiver of degree $d$} over $\Lambda$ is a quiver $Q$ over $\Lambda$ such that, for all $\lambda,\mu\in \Lambda$ (not necessarily distinct), the set $Q(\lambda,\mu)$ has cardinality $d$.
\end{definition}
For a quiver $Q$, let $Q^{rev}$ be the quiver with reversed arrows, and let the \textit{double} $\mathrm{D}Q$ be the quiver with same vertices as $Q$, and arrows formed by the disjoint union of the arrows of $Q$ and of $Q^{rev}$. A \textit{connected component} of a quiver $Q$ will be, for us, a subquiver $R$ of $Q$ that is maximal under the property that any two of its vertices are connected by a path in $\mathrm{D}R$.
\begin{definition}\label{homogeneous_quiver_of_weight_n}
We say that a quiver is \emph{homogeneous of weight $n$} if all its connected components $\Kk_i$ are complete quivers of degrees $d_i$ respectively, and the product $d_i |\Obj(\Kk_i)| = n$ is constant.
\end{definition}
Notice that, fixed a weight $n$, the shape of a homogeneous quiver of weight $n$ is uniquely determined by a string of integers $\set{N_s}_s$, where $N_s$ is the number of connected components $\Kk$ with $|\Obj(\Kk)|= s$, and $1\le s\le n$ is an integer.
\subsection{Groupoids and left semiloopoids}
\begin{definition}
    A \emph{groupoid} over $\Lambda$ is the datum of a quiver $\Gg$ over $\Lambda$; a morphism of quivers $m_\bullet\colon \Gg\ot \Gg\to \Gg$, thus defining a binary operation $x\ot y\mapsto x\bullet y$; and a family of loops $\set{\mathbf{1}_\lambda}_{\lambda\in\Lambda}$, where $\mathbf{1}_\lambda$ is a loop on $\lambda$; such that
    \begin{enumerate}
        \item the operation $\bullet$ is associative;
        \item for all $x\in\Gg$, one has $x\bullet \mathbf{1}_{\target(x)} = \mathbf{1}_{\source(x)}\bullet x = x$;
        \item for all $x\in \Gg(\lambda,\mu)$, there exists $x^{-\bullet}\in \Gg(\mu,\lambda)$ such that $x\bullet x^{-\bullet} = \mathbf{1}_{\source(x)}$ and $x^{-\bullet}\bullet x = \mathbf{1}_{\target(x)}$.
    \end{enumerate}
    A \emph{weak morphism of groupoids} $f\colon \Gg\to \Gg'$ is a weak morphism $f= (f^1, f^0)\colon \Gg\to \Gg'$ satisfying $f^1(x\bullet_\Gg y) = f^1(x)\bullet_{\Gg'}f^1(y)$ and $f^1(\mathbf{1}_\lambda) = \mathbf{1}_{f^0(\lambda)}$. A \emph{morphism of groupoids over $\Lambda$} is defined analogously.
\end{definition}
Notice that the \emph{unit} $\mathbf{1}_\lambda$ on the vertex $\lambda$ and the \emph{inverse} $x^{-\bullet}$ of $x$ are necessarily unique. A group can be interpreted as a groupoid over a singleton.

It is immediate that the underlying quiver of a groupoid must be a disjoint union of connected components $\Kk_i$, where each component is a complete quiver of some degree $d_i$.

The following definition is well known\footnote{The definition of quasigroups is quite old, and it actually captures the idea of \textit{latin squares}, that were already studied independently by Choi Seok-Jeong and by Leonhard Euler in the 18\textsuperscript{th} Century \cite{colbourn2006handbook}. Left quasigroups also appear in multiple places: we here refer to Stanovský's thesis \cite{stanovsky2004left}, but the reader should be wary that Stanovský uses the term ``groupoid'' with a different meaning than ours.}.
\begin{definition}\label{left_quasigroup}
    A \emph{left quasigroup} $(A,\bullet)$ is a set $A$ with a binary operation $\bullet$ such that the \emph{left division} is well defined, i.e., such that the left-multiplication $x\bullet \blank$ is bijective for all $x\in A$. We denote by $x\backslash y$ the unique element such that $x\bullet (x\backslash y) = y$. We say that $A$ is a \emph{left associative quasigroup} if moreover $\bullet$ is associative.
\end{definition}
The operation of taking a \emph{total} algebraic structure, and giving its equivalent with \emph{partially defined} operations (for instance groups become groupoids, algebras become algebroids\ldots), is increasingly common in the literature with the convenient---albeit not so evocative---name of \emph{oidification}; see e.g.\@ \cite{jonsson2017poloids, orchard2020unifying}. 

The following notion is an oidification of left (associative) quasigroups, and will appear very natural after Theorem \ref{groupoid_structure_on_Q(A)}. Notice that an oidification of quasigroups appears in \cite[Definition 2.2]{quasigroupoidsALONSOALVAREZ}, although it is not given in the quiver-theoretic formalism. The name \emph{quasigroupoid} is already used for a different structure, and we prefer to build on Grabowski's definition of \emph{semiloopoid} \cite{grabowski2016loopoids}.
\begin{definition}\label{def:left-semiloopoid}
    A \emph{left semiloopoid} over $\Lambda$ is the datum of a quiver $\Gg$ over $\Lambda$, and a morphism of quivers $m_\bullet\colon \Gg\ot \Gg\to \Gg$, thus defining a binary operation $x\ot y\mapsto x\bullet y$; such that, for all $x\in\Gg(\lambda,\mu)$, the multiplication $x\bullet \blank$ is a bijective map $\Gg(\mu,\Lambda)\to \Gg(\lambda,\Lambda)$. For $y\in \Gg(\lambda,\Lambda)$, the unique $z\in \Gg(\mu,\Lambda)$ such that $x\bullet z = y$ is denoted by $x\backslash y$.

    A \emph{left associative semiloopoid} is a left semiloopoid $(\Gg,\bullet)$ with $\bullet$ associative.

    A left (associative) semiloopoid is a \emph{left unital} (\emph{associative}) \emph{semiloopoid} if, for all $\lambda\in \im(\target)$, there exists a loop $\mathbf{1}_\lambda$ on $\lambda$ such that $x\bullet \mathbf{1}_\lambda = x$ and $\mathbf{1}_\lambda \bullet y = y$ for all $x\in \Gg(\Lambda,\lambda)$, $y\in \Gg(\lambda,\Lambda)$.

A \emph{weak morphism of left semiloopoids} $f\colon \Gg\to \Gg'$ is a weak morphism $f= (f^1, f^0)\colon \Gg\to \Gg'$ satisfying $f^1(x\bullet_\Gg y) = f^1(x)\bullet_{\Gg'}f^1(y)$. A \emph{morphism of left semiloopoids over $\Lambda$} is defined analogously.
\end{definition}
\begin{remark}
    If $\Gg$ is a left associative semiloopoid, then every vertex $\lambda = \target(x)\in \im(\target)$ has a loop $x\backslash x$, whence $\im(\target)\subseteq \im(\source)$, and thus $\Lambda = \im(\source)$. 
    
    Observe that a left associative semiloopoid $\Gg$ is unital if and only if, for all $\lambda\in\im(\target)$, the loop $x\backslash x$ on $\lambda$ is independent of $x\in \Gg(\Lambda,\lambda)$. Indeed $x\bullet (x\backslash x)\bullet y = x\bullet y$ implies $(x\backslash x)\bullet y = y$ for all $y$. Therefore, if $\Gg$ is moreover unital, one must have $\mathbf{1}_\lambda = x\backslash x$ for all $x\in \Gg(\Lambda,\lambda)$. Conversely, if $x\backslash x$ is independent of $x$, for all $y\in\Gg(\Lambda,\lambda)$ one also has $y\bullet (x\backslash x ) = y\bullet (y\backslash y) = y$, thus $\mathbf{1}_\lambda = x\backslash x$ is a bilateral unit, and $\Gg$ is unital.
\end{remark}
\begin{example}
    Let $\Gg$ be the quiver
    \vspace{-2.5em}
    \[ \begin{tikzcd}
        \lambda \ar[r, bend left=10, "x"above]\ar[r, bend right = 10, "y"below]& \mu \ar[loop, out=-30, in=30, looseness=4, "{x\backslash x}"right]\ar[loop, out=-40, in=40, looseness=12, "{y\backslash y}"right]
    \end{tikzcd} \vspace{-2.5em}\]    
    with multiplication \begin{align*}
    &x\bullet (x\backslash x) = x,\quad y\bullet (y\backslash y) = y, \quad (x\backslash x)\bullet (y\backslash y)= (y\backslash y),\\ &(y\backslash y)\bullet (x\backslash x)= (x\backslash x),\quad (x\backslash x)\bullet (x\backslash x)= (x\backslash x),\quad (y\backslash y)\bullet (y\backslash y)= (y\backslash y).\end{align*} Observe that this is a left associative semiloopoid, and that both $x\backslash x$ and $y\backslash y$ are left units on $\mu$, but $\Gg$ is not unital.
\end{example}

Let $\Gg$ be a left unital associative semiloopoid. We naturally have a decomposition $\Lambda = \Lambda_{\mathrm{in}}\sqcup \Lambda_{\mathrm{un}}$, where $\Lambda_{\mathrm{un}}:= \im(\target)$ will be called the set of \emph{unital vertices}, and $\Lambda_{\mathrm{in}}:= \Lambda\smallsetminus \Lambda_{\mathrm{un}}$ will be called the set of \emph{initial vertices}.

\begin{remark}
	A left unital associative semiloopoid $\Gg$ over $\Lambda = \Lambda_{\mathrm{un}}$ is just a groupoid. In particular, for every left unital associative semiloopoid $\Gg$, the subquiver $\Gg(\Lambda_{\mathrm{un}}, \Lambda_{\mathrm{un}})$ is a groupoid.
\end{remark}
As a consequence of Theorem \ref{groupoid_structure_on_Q(A)} we shall get many examples of left unital associative semiloopoids that are not groupoids.

Observe that every morphism $f\colon \Gg\to \Gg'$ of left unital associative semiloopoids respects this decomposition: $f^1(\mathbf{1}_\lambda) = f^1(x\backslash x)$ is the unique element such that $f^1(x)\bullet' f^1(x\backslash x) = f^1(x\bullet (x\backslash x)) = f^1(x)$, and hence $f^1(x\backslash x) = f^1(x)\backslash' f^1(x) = \mathbf{1}'_{f^0(\lambda)}$. In particular, $f^0(\Lambda_{\mathrm{un}})\subseteq \Lambda'_{\mathrm{un}}$.
 
Dually, one may define \emph{right (unital and/or associative) semiloopoids} on a set $\Lambda = \im(\source)\sqcup \Lambda_{\mathrm{fin}}$, where the elements of $\Lambda_{\mathrm{fin}}$ are exactly the vertices with no outgoing arrow, and we call them \emph{final vertices}.
\begin{lemma}\label{lemma:semiloopoid_bijections}
Let $\Gg$ be a left unital associative semiloopoid. Then for all $\lambda\in \Lambda_{\mathrm{in}}$, and for all $\mu, \nu\in \Lambda_{\mathrm{un}}$ in the same connected component, such that $\Gg(\lambda,\nu)\neq \emptyset$, there are bijections $\Gg(\lambda,\mu)\cong \Gg(\mu, \nu) \cong Q(\lambda, \nu)$.
\end{lemma}
\begin{proof}
Let $y\in \Gg(\mu,\nu)$. Then $y$ is invertible, and the map $\blank\bullet y$ is a bijection $\Gg(\lambda,\mu)\to \Gg(\lambda,\nu)$.

Now we prove the bijection $\Gg(\mu,\nu)\cong \Gg(\lambda,\nu)$. Since there is an arrow $z\in\Gg(\lambda,\nu)$, and there is an (invertible) arrow $y\in \Gg(\mu,\nu)$, let $x:= z\bullet y^{-\bullet}$. It is clear that $x\bullet \blank$ is the desired bijection.
\end{proof}
\subsection{Yang--Baxter maps and braidings}
\begin{definition}
    Let $Q$ be a quiver over $\Lambda$. A (\emph{quiver-theoretic}) \emph{Yang--Baxter map} on $Q$ is a morphism $\sigma\colon Q\ot Q\to Q\ot Q$, satisfying the \emph{Yang--Baxter equation}\footnote{We call ``Yang--Baxter equation'' what should more correctly be called ``braid relation''. This mash-up of terminology is very customary for those who work with set-theoretic or quiver-theoretic solutions. However the reader is warned that, in the Hopf-theoretic and Lie-theoretic settings, the same habit can generate confusion.} 
    \begin{equation}\label{eq:ybe}\tag{\sc ybe}
        (\sigma\ot \id)(\id \ot \sigma)(\sigma\ot \id)= (\id \ot \sigma) (\sigma\ot \id)(\id \ot \sigma).
    \end{equation}
    We usually write in components $\sigma(x\ot y) = (x\rightharpoonup y)\ot (x\leftharpoonup y)$. A Yang--Baxter map $\sigma$ is \emph{left non-degenerate} if $x\rightharpoonup\blank\colon Q(\target(x), \Lambda)\to Q(\source(x),\Lambda)$ is a bijection for all $x$. It is \emph{right non-degenerate} if $\blank\leftharpoonup y\colon Q(\Lambda,\source(y))\to Q(\Lambda, \target(y))$ is a bijection for all $y$. It is \emph{non-degenerate} if it is both left and right non-degenerate. A Yang--Baxter map $\sigma$ is \emph{involutive} if $\sigma^2 = \id$.
\end{definition}
The \textsc{ybe} can be interpreted as a \emph{cube rule}; see Figure \ref{fig:closure_of_cubes}. This interpretation is pervasive in Dehornoy \textit{et al.\@} \cite{dehornoy2015foundations}.
\begin{figure}[t]
    \centering
    \begin{tikzpicture}
    \filldraw[fill=gray, draw opacity =0, fill opacity =0.2] (0,0.5)--(0,1.5)--(1,1.5)--(1,0.5)--(0,0.5);
    \draw[-Stealth] (0,1.5) to  node[sloped, above]{$x$} (1,1.5);
    \draw[-Stealth] (1,1.5) to  node[sloped, above]{$y$} (1,0.5);
    \draw[-Stealth] (0,1.5) to  node[sloped, below]{$x\rightharpoonup y$} (0,0.5);
    \draw[-Stealth] (0,0.5) to  node[sloped, below]{$x\leftharpoonup y$} (1,0.5);
    \draw[-Stealth] (0.2,1) arc[radius=.3, start angle=180, end angle=-90];
    \draw[-Stealth] (4.5,2) to (4.5,0.5);
    \draw[-Stealth] (3,0.5) to (4.5,0.5);
     \draw[-Stealth] (4.5,0.5) to (5,0);
    \filldraw[fill=gray, draw opacity =0, fill opacity =0.4] (3,0.5)--(3,2)--(4.5,2)--(5,1.5)--(5,0)--(3.5,0)--(3,0.5);
    \draw[-Stealth, thick] (3,2) to node[above]{$x$} (4.5,2);
    \draw[-Stealth] (3,2) to (3,0.5);
    \draw[-Stealth] (3,0.5) to (3.5,0);
    \draw[-Stealth] (3.5,0) to (5,0);
    \draw[-Stealth] (3.5,1.5) to (3.5,0);
    \draw[-Stealth] (3,2) to (3.5,1.5);
    \draw[-Stealth] (3.5,1.5) to (5,1.5);
    \draw[-Stealth, thick] (5,1.5) to node[right]{$z$} (5,0);
    \draw[-Stealth, thick] (4.5,2) to node[above]{$y$} (5,1.5);
    \end{tikzpicture}
    \caption{The \textsc{ybe} for $\sigma$ can be seen as the closure of cubes, where each \emph{oriented} face represents the application of $\sigma$. If $\sigma$ is moreover involutive, the orientation of the faces is superfluous.}
    \label{fig:closure_of_cubes}
\end{figure}

We recall the following definition from Andruskiewitsch \cite{andruskiewitsch2005quiver}. Here, for an endomap $f$ of $\Gg^{\otimes 2}$, we use the well-established notation $f_{12} =f\ot \id_\Gg$, $f_{23} = \id_{\Gg}\ot f$.
\begin{definition}\label{braided_groupoid}
    A \emph{pre-braided groupoid} is the datum of a groupoid $\Gg$, $\Lambda:= \Obj(\Gg)$, with multiplication $m\colon \Gg\otimes \Gg\to \Gg$ and family of units $\set{\mathbf{1}_\lambda}_{\lambda\in\Lambda}$; and of a morphism of $\mathsf{Quiv}_\Lambda$ $\sigma\colon \Gg\otimes \Gg\to \Gg\otimes \Gg$, $\sigma(x\ot y) = (x\rightharpoonup y)\ot( x\leftharpoonup y)$, satisfying the following properties for all $x\ot y\ot z$:
    \begin{align}
       &\label{bg1}\tag{\sc bg1}\sigma (x\ot \mathbf{1}_{\target(x)})=\mathbf{1}_{\source(x)}\ot  x;\\
       &\label{bg2}\tag{\sc bg2}\sigma(\mathbf{1}_{\source(x)}\ot x)= x\ot \mathbf{1}_{\target(x)};\\
        &\label{bg3}\tag{\sc bg3}\sigma m_{23} = m_{12}\sigma_{23}\sigma_{12},\\\nonumber &\text{i.e.\@ }x\rightharpoonup yz = (x\rightharpoonup y)((x\leftharpoonup y)\rightharpoonup z)\text{ and }x\leftharpoonup yz = (x\leftharpoonup y)\leftharpoonup z;\\
        &\label{bg4}\tag{\sc bg4} \sigma m_{12} = m_{23}\sigma_{12}\sigma_{23},\\&\nonumber \text{i.e.\@ }xy\leftharpoonup z = (x\leftharpoonup(y\rightharpoonup z))(y\leftharpoonup z)\text{ and }xy\rightharpoonup z = x\rightharpoonup(y\rightharpoonup z);\\
        &\label{braided-commutative}\tag{\sc bg5} m\sigma = m.
    \end{align}
The morphism $\sigma$ is called a \emph{pre-braiding}. A pre-braided groupoid $(\Gg,\sigma)$ is said to be \emph{braided} if $\sigma$ is an isomorphism; and in this case, $\sigma$ is called a \emph{braiding}. We denote by $\mathsf{BrGpd}_\Lambda$ the category of braided groupoids over $\Lambda$, endowed with morphisms $f$ of groupoids such that $f\ot f$ intertwines the two braidings.
\end{definition}
Conditions \eqref{bg3}, \eqref{bg4} are also called the \emph{hexagonal axioms}, while \eqref{braided-commutative} is the \emph{braided-commutativity} of $m$ with respect to $\sigma$. 

It is well known that pre-braided groupoids produce solutions to the \textsc{ybe} \cite{andruskiewitsch2005quiver}.
The solution corresponding to a pre-braided groupoid is left and right non-degenerate, and it is moreover bijective if the groupoid is braided.
\begin{remark}
    It is well known that, in general, a groupoid $\Gg$ may have solutions $\sigma\colon \Gg\ot \Gg\to \Gg\ot \Gg$ to the \textsc{ybe} that are not braidings. This is already true when $\Lambda$ is a singleton.
\end{remark}
The following definition is an oidification of skew (left) braces, discovered by Sheng, Tang, and Zhu \cite{sheng2024postgroupoidsquivertheoreticalsolutionsyangbaxter}. Notice that the name \emph{skew bracoid} was already used in the literature, by Martin-Lyons and Truman \cite{martin2024skew}, for a completely different structure which is relevant in Hopf--Galois theory and in the theory of the Yang--Baxter equation, but has no pretense of being an oidification of skew braces.
\begin{definition}[Sheng, Tang, and Zhu \cite{sheng2024postgroupoidsquivertheoreticalsolutionsyangbaxter}] A \emph{skew bracoid} over $\Lambda$ is the datum $(\Gg,\set{\cdot_\lambda}_{\lambda\in\Lambda},\bullet, \set{\mathbf{1}_\lambda}_{\lambda\in\Lambda})$, of a groupoid $(\Gg,\bullet, \set{\mathbf{1}_\lambda}_{\lambda\in\Lambda})$ over $\Lambda$ and, for all $\lambda\in\Lambda$, of a group operation $\cdot_\lambda$ on $\Gg(\lambda,\Lambda)$; satisfying the following compatibility for all $x\ot y,\, x\ot z\in \Gg\ot \Gg$:
    \begin{equation}\label{qtbc}\tag{\sc b{\tiny oid}c} x\bullet (y\cdot_{\target(x)} z) = (x\bullet y)\cdot_{\source(x)} x^{-1_{\source(x)}}\cdot_{\source(x)} (x\bullet z), \end{equation}
    where $x^{-1_{\source(x)}}$ denotes the inverse of $x$ in $(\Gg(\source(x),\Lambda), \cdot_{\source(x)})$. Notice that the bracoid compatibility \eqref{qtbc} implies that $\mathbf{1}_\lambda$ is the unit of $(\Gg(\lambda,\Lambda), \cdot_\lambda)$ for all $\lambda\in\Lambda$.
\end{definition}
As a generalisation of post-groups (see \cite{postgroups}), the following definition has also been given.
\begin{definition}[{Sheng, Tang, and Zhu \cite{sheng2024postgroupoidsquivertheoreticalsolutionsyangbaxter}}] \label{def:postgroupoid}A \textit{post-groupoid} over $\Lambda$ is the datum $(p\colon \Gg\to \Lambda, \set{\cdot_\lambda}_{\lambda\in\Lambda}, \set{1_\lambda}_{\lambda\in\Lambda}, \Phi, \triangleright)$ of a group bundle $p\colon \Gg\to \Lambda$, where the group structures on the loops $\Gg_\lambda$ are denoted by $(\Gg_\lambda, \cdot_\lambda, 1_\lambda)$; of a surjective map $\Phi\colon \Gg\to \Lambda$; and of a map $\triangleright\colon \Gg \fibre{\Phi}{p}\Gg\to \Gg$; satisfying
	\begin{enumerate}
		\item $\Phi(1_\lambda) = \lambda$ for all $\lambda\in\Lambda$;
		\item $p(g\triangleright h) = p(g)$ for all $g\in \Gg$, $h\in \Gg_{\Phi(g)}$;
		\item $\Phi(g\cdot_{p(g)} (g\triangleright h)) = \Phi(h)$ for all $g\in \Gg$, $h\in \Gg_{\Phi(g)}$;
		\item $g\triangleright (h\cdot_{\Phi(g)} k) = (g\triangleright h)\cdot_{p(g)}(g\triangleright k)$ for all $g\in \Gg$, $h,k\in \Gg_{\Phi(g)}$;
		\item $g\triangleright (h\triangleright k) = (g\cdot_{p(g)} (g\triangleright h))\triangleright k$ for all $g\in\Gg$, $h\in \Gg_{\Phi(g)}$, $k\in \Gg_{\Phi(h)}$;
		\item the maps $g\triangleright\blank \colon \Gg_{\Phi(g)}\to \Gg_{p(g)}$ are bijective for all $g$.
	\end{enumerate}
\end{definition}
The following result motivates the above definitions.
\begin{proposition}[{Andruskiewitsch \cite{andruskiewitsch2005quiver}, Sheng, Tang, and Zhu \cite{sheng2024postgroupoidsquivertheoreticalsolutionsyangbaxter}}] Given a groupoid $(\Gg,\bullet, \set{\mathbf{1}_\lambda}_{\lambda\in\Lambda})$ over $\Lambda$, the following data are equivalent:
\begin{enumerate}
    \item a braiding on $\Gg$;
    \item a matched pair of actions on a groupoid, as defined in Mackenzie \cite{mackenzie};
    \item a family of group structures $\set{(\Gg(\lambda,\Lambda), \cdot_\lambda)}_{\lambda\in\Lambda}$ that yield a skew bracoid on $\Gg$;
    \item a post-groupoid $(\source\colon \Gg\to \Lambda, \set{\cdot_\lambda}_{\lambda\in\Lambda}, \set{\mathbf{1}_\lambda}_{\lambda\in\Lambda}, \target, \triangleright)$ such that $x\bullet y = x\cdot_{\source(x)}(x\triangleright y)$ for all $x\ot y\in \Gg{\ot 2}$;
    \item a $1$-cocycle groupoid datum, as defined by Andruskiewitsch \cite{andruskiewitsch2005quiver}.
\end{enumerate}
\end{proposition}
In the previous proposition, we shall exclusively be interested in the equivalences between (\textit{i})--(\textit{iv}), which are also isomorphisms of categories; see \cite{andruskiewitsch2005quiver, sheng2024postgroupoidsquivertheoreticalsolutionsyangbaxter}. The correspondence of the other data with (\textit{v}) is just an equivalence of categories.

We give an explicit outline of the correspondences between (\textit{i})--(\textit{iv}).
\begin{enumerate}
	\item[(\textit{i})--(\textit{ii})] If $\sigma\colon x\ot y \mapsto (x\rightharpoonup y)\ot (x\leftharpoonup y)$ is a braiding, then $(\rightharpoonup,\leftharpoonup)$ is a matched pair of actions. Every braiding can be constructed in this way.
	\item[(\textit{ii})--(\textit{iii})] Given the actions $(\rightharpoonup,\leftharpoonup)$, we define the bundle of group structures as $x\cdot_\lambda y:= x\bullet (x^{-1}\rightharpoonup y)$, for $x,y\in \Gg(\lambda,\Lambda)$. Conversely, given a skew bracoid on $\Gg$, we define $x\rightharpoonup y = x^{-1_{\source(x)}} \cdot_{\source(x)} (x\bullet y)$ and $x\leftharpoonup y = (x\rightharpoonup y)^{-\bullet}\bullet x\bullet y$, for $x\ot y\in \Gg^{\ot 2}$. 
	\item[(\textit{iii})--(\textit{iv})] The bundle $p\colon \Gg\to \Lambda$ in (\textit{iv}) corresponds to the source map $\source\colon \Gg\to \Lambda$ in (\textit{iii}). The surjective map $\Phi\colon \Gg\to \Lambda$ in (\textit{iv}) corresponds to the target map $\target\colon \Gg\to \Lambda$ in (\textit{iii}). Each set of outgoing arrows $\Gg(\lambda,\Lambda)$ in $\Gg\twomapsright{\source}{\target}\Lambda$ corresponds to a set of loops in $p\colon \Gg\to\Lambda$. The group bundle operation $\cdot_\lambda$ on the set of loops $p^{-1}(\lambda)$ in (\textit{iv}) will correspond precisely to the operation $\cdot_\lambda$ on $\Gg(\lambda,\Lambda)$ given in (\textit{iii}). Since now $\Gg\fibre{\Phi}{p}\Gg$ in (\textit{iv}) corresponds to $\Gg\ot \Gg$ in (\textit{iii}), we can define $\triangleright$ as $\rightharpoonup$. The groupoid structure $\bullet$ is obtained from the post-groupoid structure by setting $x\bullet y  = x\cdot_{p(x)} (x\triangleright y)$ for all $(x,y)\in \Gg\fibre{\Phi}{p}\Gg$ (this is called the \textit{Grossman--Larson groupoid}, see \cite[Theorem 2.16]{sheng2024postgroupoidsquivertheoreticalsolutionsyangbaxter}).
\end{enumerate}

In view of a slight generalisation of the correspondence, we give the following definition.
\begin{definition}\label{def:braided_semiloopoid}
    A \emph{pre-braided left unital associative semiloopoid} is the datum of a left unital associative semiloopoid $\Gg$, with $\Lambda = \Lambda_{\mathrm{in}}\sqcup \Lambda_{\mathrm{un}} = \Obj(\Gg)$, endowed with the multiplication $m\colon \Gg\ot \Gg\to \Gg$ and the family of right units $\set{\mathbf{1}_\lambda}_{\lambda\in\Lambda_{\mathrm{un}}}$; and of a morphism of $\mathsf{Quiv}_\Lambda$ $\sigma\colon \Gg\ot \Gg\to \Gg\ot \Gg$, $\sigma(x\ot y) = (x\rightharpoonup y)\ot (x\leftharpoonup y)$, satisfying \eqref{bg1} and \eqref{bg2} for all $x\in \Gg(\Lambda_{\mathrm{un}}, \Lambda_{\mathrm{un}})$, and satisfying \eqref{bg3}--\eqref{braided-commutative} for all $x\ot y\ot z$. We call $\sigma$ a \emph{pre-braiding} on $\Gg$. If moreover $\sigma $ is bijective, we call it a \emph{braiding}, and we say that $\Gg$ is \emph{braided}.
\end{definition}
\begin{lemma}
    Pre-braided left unital associative semiloopoids produce solutions to the \textsc{ybe}. 
\end{lemma}
\begin{proof} Given $x\ot y\ot z \in \Gg^{\ot 3}$, let $\sigma_{12}\sigma_{23}\sigma_{12}(x\ot y\ot z) = a\ot b\ot c$ and let $\sigma_{12}\sigma_{23}\sigma_{12}(x\ot y\ot z) = a'\ot b'\ot c'$. In a similar fashion as \cite[Theorem 1]{LYZ}, using \eqref{braided-commutative} one proves $a\bullet b\bullet c = a'\bullet b'\bullet c'$. Moreover, $c = c'$ and $a = a'$ follow from \eqref{bg3} and \eqref{bg4} respectively. Since $a\bullet\blank$ is invertible, this implies $b\bullet c = b'\bullet c$. Finally, since both the sources and the targets of $b$ and $b'$ lie in $\Lambda_{\mathrm{un}}$, and since $\Gg(\Lambda_{\mathrm{un}},\Lambda_{\mathrm{un}})$ is a groupoid, we conclude $b = b'$.
\end{proof}
Dually, one may give a definition of \emph{braided right unital associative semiloopoids}, and prove that they also produce solutions of the \textsc{ybe}.
\begin{remark}
Let $(\Gg,\sigma)$ be a pre-braided left unital associative semiloopoid. The unital part of $\Gg$ is a grupoid, and $\sigma$ restricts to a pre-braiding of groupoids on $\Gg(\Lambda_{\mathrm{un}}, \Lambda_{\mathrm{un}})$.
\end{remark}

\subsection{Groupoids of pairs}
Let $\Lambda$ be a nonempty set. The complete quiver of degree $1$ over $\Lambda$ is called the \emph{groupoid of pairs} over $\Lambda$ (see e.g.\@ \cite[Example 1.11]{introGroupoids}). We denote it by $\hat{\Lambda}$. For all $a,b\in \Lambda$, we denote by $[a,b]$ the unique arrow from $a$ to $b$. For $a_0, \dots, a_n\in \Lambda$, the unique path that touches the vertices $a_0,\dots, a_n$ in this order is denoted by $[a_0, \dots, a_n] := [a_0, a_1]\ot [a_1, a_2]\ot \dots \ot [a_{n-1}, a_n]$.
\begin{remark}\label{rem:indeed_are_groupoids}
    Groupoids of pairs are indeed groupoids, with the multiplication $\bullet$ being defined (in the only possible way) by $[a,b]\bullet [b,c] = [a,c]$.
\end{remark}
We shall use the expressions ``groupoid of pairs'' and ``complete quiver of degree $1$'' as synonyms, although they only are so up to isomorphism. 
\begin{remark}
    If $\Gg$ is the groupoid of pairs over $\Lambda$, a morphism $\sigma\colon \Gg\ot \Gg\to \Gg\ot \Gg$ is uniquely determined by a ternary operation $\langle\blank,\blank,\blank\rangle$ on $\Lambda$, via the
    \begin{equation}\label{eq:sigma_ternary} \sigma([a,b,c]) = [a,\langle a,b,c\rangle, c]. \end{equation}
    This is the ternary operation appearing in Proposition \ref{prop:braiding-iff-ternary-iff-group}.
\end{remark}

\subsection{Heaps} The definition of \emph{heap} was originally given by Baer \cite{baer1929einfuhrung} with the German name \textit{(die) Schar}, plur.\@ \textit{Scharen}, generalising Pr\"ufer \cite{prufer1924theorie}, and then further generalised by Wagner \cite{wagner1953} who translated the name in Russian as груда. We here refer to Hollings and Lawson \cite{hollings2017wagner} and Brzezi\'nski \cite{BRZEZINSKITrussesParagons} for the current notion and the main results.
\begin{definition}
    A \emph{heap} is the datum of a set $\Lambda$ and a ternary operation $\langle\blank,\blank,\blank\rangle$ on $\Lambda$ satisfying, for all $a,b,c,d\in \Lambda$, the following conditions:
    \begin{align}
        &\label{maltsev_1}\tag{\sc m1} \langle a,b,b\rangle = a,\\
        &\label{maltsev_2}\tag{\sc m2} \langle a,a,b\rangle = b,\\
        &\label{associativity}\tag{\sc a} \langle a,b,\langle c,d,e\rangle\rangle = \langle \langle a,b,c\rangle, d,e\rangle.
    \end{align}
    A heap is called \emph{abelian} if $\langle a,b,c\rangle = \langle c,b,a\rangle$ for all $a,b,c\in\Lambda$. 
\end{definition}
\begin{remark}\label{rem:heaps_and_groups}
   There is a well known correspondence between groups and pointed heaps; see \cite{baer1929einfuhrung, breaz2024heaps,  BRZEZINSKITrussesParagons, wagner1953}. If $\Lambda$ is a heap, and $\zeta\in\Lambda$ is a distinguished element, then the binary operation $a*_\zeta b := \langle a,\zeta, b\rangle$ defines a group structure on $\Lambda$ with unit $\zeta$. The several group structures obtained by choosing a unit $\zeta$ are all isomorphic with each other, and the isomorphism $(\Lambda,*_\zeta)\to (\Lambda,*_\xi)$ is given by $\blank *_\xi \zeta^{-\xi}$ where $(\blank)^{-\xi}$ denotes the inverse with respect to $*_\xi$. Conversely, if $(A,\cdot)$ is a group, then $\langle a,b,c\rangle:= a b^{-1}c$ defines a heap structure.
\end{remark}
\begin{remark}\label{rem:a-iff-a1a2}
     Suppose given a set $\Lambda$ and a ternary operation $\langle\blank,\blank,\blank\rangle$ on $\Lambda$ satisfying \eqref{maltsev_1} and \eqref{maltsev_2}. It is known (see \cite[Proposition 7]{BournGranJacqmin}) that this is a heap if and only if the following axioms are satisfied for all $a,b,c,d\in\Lambda$:
    \begin{align}
        \label{associativity_1}\tag{\sc a1}&\langle a,b,d\rangle = \langle \langle a,b,c\rangle, c, d\rangle,\\
\label{associativity_2}\tag{\sc a2}&\langle a,c,d\rangle = \langle a, b, \langle b,c,d\rangle\rangle.
    \end{align}
\end{remark}
Groupoids of pairs are relevant to us for the following reason.
\begin{proposition}[{see \cite[Theorem 7.13]{ferrishibu}}] \label{prop:braiding-iff-ternary-iff-group}
    Let $\Lambda$ be a set, and $\Gg = \hat{\Lambda}$ the associated groupoid of pairs. Then, the following data are equivalent:
    \begin{enumerate}
        \item[(\textsc{a}\textit{i})] a braiding on $(\Gg,\bullet)$;
        \item[(\textsc{a}\textit{ii})] a heap structure on $\Lambda$.
    \end{enumerate}
    As a consequence, the following data are also equivalent:
     \begin{enumerate}
        \item[(\textsc{b}\textit{i})] a braiding on $(\Gg,\bullet)$, and a distinguished vertex $\zeta\in\Lambda$;
        \item[(\textsc{b}\textit{ii})] a heap structure on $\Lambda$, and a distinguished element $\zeta\in\Lambda$;
        \item[(\textsc{b}\textit{iii})] a group structure on $\Lambda$.
    \end{enumerate}
    Via this correspondence, involutive braidings correspond to abelian group structures and to abelian heap structures.
\end{proposition}
\begin{proof}
    One immediately has the equivalences\begin{align*} &\eqref{maltsev_1}\iff\eqref{bg1}, && \eqref{maltsev_2}\iff\eqref{bg2},\\
    &\eqref{associativity_1}\iff\eqref{bg3}, && \eqref{associativity_2}\iff\eqref{bg4}.\end{align*} 
    Observe that \eqref{braided-commutative} is always satisfied, because $m \sigma([a,b,c])$ and $m([a,b,c])$ are both forced to be $[a,c]$. This proves (\textsc{a}\textit{i})$\iff$(\textsc{a}\textit{ii}). The rest of the statement follows from Remark \ref{rem:heaps_and_groups}.
\end{proof}
Solutions of the \textsc{ybe} on a groupoid of pairs are already considered in \cite{matsumotoshimizu, shibukawa2007invariance}, and they are called \emph{principal homogeneous solutions}.
\section{Quiver-theoretic skew braces}\noindent \label{sec:qtsb}In this section, we define \emph{quiver-theoretic skew braces} as a slight generalisation of skew bracoids, aimed at comprising the case of non zero-symmetric dynamical skew braces. We prove that quiver-theoretic skew braces produce left non-degenerate braided left unital associative semiloopoids. The computations are very direct, and reminiscent of \cite[Theorem 3.1]{guarnieri2017skew} and other similar results.
\begin{definition}\label{def:qtsb}
    A \textit{quiver-theoretic skew brace} $(\Gg,\set{\cdot_\lambda}_{\lambda\in\Lambda},\bullet, \set{\mathbf{1}_\lambda}_{\lambda\in\Lambda})$ over $\Lambda$ is the datum of a left unital associative semiloopoid $(\Gg,\bullet, \set{\mathbf{1}_\lambda}_{\lambda\in\Lambda_{\mathrm{un}}})$ over $\Lambda = \Lambda_{\mathrm{in}}\sqcup\Lambda_{\mathrm{un}}$ and, for all $\lambda\in\Lambda$, of a group operation $\cdot_\lambda$ on $\Gg(\lambda,\Lambda)$; satisfying the bracoid compatibility \eqref{qtbc}. 
    
    Notice that the compatibility \eqref{qtbc} implies that $\mathbf{1}_\lambda$ is the unit of $(\Gg(\lambda,\Lambda), \cdot_\lambda)$ for all $\lambda\in\Lambda_{\mathrm{un}}$.
    Reprising \cite[Definition 3.4]{matsu}, a quiver-theoretic skew brace is called \emph{zero-symmetric} if $\Lambda = \Lambda_{\mathrm{un}}$, i.e., if $\Gg$ is a groupoid, i.e., if $(\Gg,\set{\cdot_\lambda}_{\lambda\in\Lambda},\bullet, \set{\mathbf{1}_\lambda}_{\lambda\in\Lambda})$ is a skew bracoid.

    A \emph{quiver-theoretic brace}, or quiver-theoretic skew brace \emph{of abelian type}, is a quiver-theoretic skew brace such that all the group operations $\cdot_\lambda$ are commutative.

    A \textit{weak morphism} $f\colon \Gg\to \Hh$ of quiver-theoretic skew braces is a weak morphism $f = (f^1, f^0)\colon \Gg\to \Hh$ which is a weak morphism of left unital associative semiloopoids from $(\Gg,\bullet_\Gg)$ to $  (\Hh,\bullet_\Hh)$, and such that $f^1|_{\Gg(\lambda,\Lambda)}^{\Hh(f^0(\lambda),\Lambda)}\colon (\Gg(\lambda,\Lambda), \cdot_\lambda)\to (\Hh(f^0(\lambda),\Lambda), \cdot_{f^0(\lambda)})$ is a group homomorphism for all $\lambda\in\Lambda$. A morphism  \emph{over $\Lambda$} is a weak morphism with $f^0 = \id_\Lambda$. Denote by $\mathsf{QTSB}_\Lambda$ the category of quiver-theoretic skew braces over $\Lambda$, with morphisms over $\Lambda$; and denote by $\mathsf{QTSB}^0_\Lambda$ its full subcategory of skew bracoids over $\Lambda$.
\end{definition}
We shall henceforth drop the subscripts from the notation (observe that the subscripts specifying the vertices can be retrieved uniquely from an expression, by the request that the expression is well defined). As usual, we omit the symbol $\cdot$ in the multiplication. We denote a quiver-theoretic skew brace by simply $\Gg$, when the other operations are understood.
\begin{remark}
A Yang--Baxter map on a quiver can be seen as a partial Yang--Baxter map on a set. Notions of \emph{partial braces} \cite{chouraquiPartial} and \emph{$\mathscr{M}$-braces} \cite{chouraquiMbraces} have been introduced by Chouraqui, and both are distinct from quiver-theoretic skew braces.
\end{remark}

\begin{lemma}
    Let $\Gg$ be a quiver-theoretic skew brace. Then:
    \begin{enumerate}
        \item $a\bullet (b^{-1}c) = a(a\bullet b)^{-1}(a\bullet c)$;
        \item $a\bullet (bc^{-1}) = (a\bullet b)(a\bullet c)^{-1}a$;
    \end{enumerate}
    for all $a\ot b,\, a\ot c\in\Gg\ot \Gg$.
\end{lemma}
\begin{proof}
    The proof is the same as for \cite[Lemma 1.7]{guarnieri2017skew}: just observe that all steps are well defined. Notice, moreover, that the proof never uses the inverse $(\blank)^{-\bullet}$, thus the result holds without the assumption of zero-symmetry.
\end{proof}
\begin{proposition}[{cf.\@ \cite[Proposition 1.9]{guarnieri2017skew}}] \label{skew-brace-equivalent-conditions}
    Let $(\Gg,\bullet)$ be a left unital associative semiloopoid over $\Lambda=\Lambda_{\mathrm{in}}\sqcup\Lambda_{\mathrm{un}}$, and for all $\lambda\in\Lambda$ let $\cdot_\lambda$ be a group operation on $\Gg(\lambda,\Lambda)$. For all $a\ot b\in \Gg\ot \Gg$, define $a\rightharpoonup b:= a^{-1}(a\bullet b)$. Then, the following are equivalent:
    \begin{enumerate}
        \item $\Gg$ is a quiver-theoretic skew brace;
        \item $(a\bullet b)\rightharpoonup c = a\rightharpoonup (b\rightharpoonup c)$ for all $a\ot b\ot c\in \Gg^{\ot 3}$;
        \item $a\rightharpoonup (bc) = (a\rightharpoonup b)(a\rightharpoonup c)$ for all $a\ot b, a\ot c\in \Gg^{\ot 2}$.
    \end{enumerate}
\end{proposition}
\begin{proof}
    We first prove the implication (\textit{i})$\implies$(\textit{ii}). Observe that $(a\bullet b)a^{-1}(a\bullet b^{-1})a^{-1} = a\bullet (bb^{-1})a^{-1} = aa^{-1} = 1$, thus $(a\bullet b)^{-1} = a^{-1}(a\bullet b^{-1})a^{-1}$. Therefore we compute
    \begin{align*}
        (a\bullet b)\rightharpoonup c& = (a\bullet b)^{-1}\cdot (a\bullet b\bullet c)\\
        & = a^{-1}(a\bullet b^{-1})a^{-1} (a\bullet b\bullet c)\\
        &= a^{-1}\left( a\bullet (b^{-1}\cdot (b\bullet c))  \right)\\
        &= a\rightharpoonup (b\rightharpoonup c),
    \end{align*}
    as desired. We now prove the implication (\textit{ii})$\implies$(\textit{iii}): from (\textit{ii}) one has
    \[ a^{-1}\left( a\bullet (b^{-1}\cdot (b\bullet c))\right)  = (a\bullet b)^{-1} (a\bullet b\bullet c),\]
    thus \[ (a\bullet b)a^{-1}(a\bullet (b^{-1}(b\bullet c))) = a\bullet b\bullet c, \]
    and now it suffices to substitute $d = b^{-1}(b\bullet c)$ to retrieve the brace compatibility: this substitution is a bijective function of $c$, because $b\bullet \blank$ is bijective. Finally, we prove (\textit{iii})$\implies$(\textit{i}): condition (\textit{iii}) reads 
    \[ a^{-1}(a\bullet (bc))  = a^{-1}(a\bullet b)a^{-1}(a\bullet c),\]
    and now it suffices to cancel $a^{-1}$, to get the brace compatibility.
    \end{proof}
\begin{remark}\label{rem:bullet-as-cdot} 
    By definition of $\rightharpoonup$, one immediately has $a\bullet b = a\cdot (a\rightharpoonup b)$ and $a\cdot b = a\bullet (a\rightharpoonup\blank)^{-1}(b)$; see \cite[Remark 1.8]{guarnieri2017skew}.
\end{remark}
\begin{theorem}\label{thm:qtsb_yields_braiding}
    Let $(\Gg,\bullet)$ be a quiver-theoretic skew brace. Then $\sigma\colon a\ot b\mapsto (a\rightharpoonup b)\ot (a\leftharpoonup b)$, with $a\rightharpoonup b:= a^{-1}(a\bullet b)$ and $a\leftharpoonup b:= (a\rightharpoonup b)\backslash (a\bullet b)$,  is a left non-degenerate braiding on $\Gg$. This is involutive if and only if $\Gg$ is a quiver-theoretic brace.
\end{theorem}
\begin{proof} Suppose given a quiver-theoretic skew brace on $\Gg$, and define $a\rightharpoonup b:= a^{-1}(a\bullet b)$ and $a\leftharpoonup b:= (a\rightharpoonup b)\backslash (a\bullet b)$ for all $a\ot b\in \Gg\ot \Gg$. As usual we omit the subscripts specifying the vertex. Condition \eqref{braided-commutative} holds by construction, and $a\rightharpoonup\blank$ is clearly bijective because $a\bullet \blank$ is. By Proposition \ref{skew-brace-equivalent-conditions} (\textit{ii}), one has $(a\bullet b)\rightharpoonup c = a\rightharpoonup(b\rightharpoonup c)$. This also implies \begin{equation}\label{eq:backslash-action}(a\rightharpoonup\blank)^{-1}(b\rightharpoonup c) = (a\backslash b)\rightharpoonup c.\end{equation} By Proposition \ref{skew-brace-equivalent-conditions} (\textit{iii}) and Remark \ref{rem:bullet-as-cdot}, one has
    \begin{align*}
        a\rightharpoonup (b\bullet c)& = a\rightharpoonup (b(b\rightharpoonup c))\\
        &= (a\rightharpoonup b)\cdot (a\rightharpoonup(b\rightharpoonup c))\\
        &= (a\rightharpoonup b)\cdot ((a\bullet b)\rightharpoonup c)\\
        \overset{(\dagger)}&{=} (a\rightharpoonup b)\bullet \left( ((a\rightharpoonup b)\backslash (a\bullet b))\rightharpoonup c \right)\\
        &= (a\rightharpoonup b)\bullet ((a\leftharpoonup b)\rightharpoonup c)
        ,
    \end{align*}
    where the equality marked with ($\dagger$) follows from Remark \ref{rem:bullet-as-cdot} and \eqref{eq:backslash-action}. By definition, $a\leftharpoonup (b\bullet c) = (a\rightharpoonup(b\bullet c))\backslash (a\bullet b\bullet c)$. Since
    \begin{align*}
        &(a\rightharpoonup (b\bullet c))\bullet (a\leftharpoonup b)\leftharpoonup c)\\ 
        &= (a\rightharpoonup b)\bullet ((a\leftharpoonup b)\rightharpoonup c)\bullet ((a\leftharpoonup b)\leftharpoonup c)\\
        &= (a\rightharpoonup b)\bullet (a\leftharpoonup b)\bullet c\\
        &= a\bullet b\bullet c,
    \end{align*}
    one has $(a\leftharpoonup b)\leftharpoonup c = a\leftharpoonup (b\bullet c)$. Finally, since $(a\bullet b)\leftharpoonup c:= ((a\bullet b)\rightharpoonup c)\backslash (a\bullet b\bullet c)$ and 
    \begin{align*}
        &((a\bullet b)\rightharpoonup c)\bullet (a\leftharpoonup (b\rightharpoonup c))\bullet (b\leftharpoonup c)\\
        &= (a\rightharpoonup (b\rightharpoonup c))\bullet (a\leftharpoonup (b\rightharpoonup c))\bullet (b\leftharpoonup c)\\
        &= a\bullet (b\rightharpoonup c)\bullet (b\leftharpoonup c)\\
        &= a\bullet b\bullet c,
    \end{align*}
    one also has $(a\bullet b)\leftharpoonup c = (a\leftharpoonup(b\rightharpoonup c))\bullet (b\leftharpoonup c)$. This concludes the proof of the hexagonal axioms. The axioms \eqref{bg1} and \eqref{bg2} on unital vertices are immediate. From the data $a\rightharpoonup b$ and $a\leftharpoonup b$, one can retrieve $a\bullet b = (a\rightharpoonup b)\bullet (a\leftharpoonup b)$, and thus $a = (a\bullet b)(a\rightharpoonup b)^{-1}$ and $b = a\backslash(a\bullet b)$. Therefore, $\sigma$ is bijective, and hence a braiding.
    
    The involutivity conditions are
    \begin{align*}
        &(a\rightharpoonup b)\rightharpoonup (a\leftharpoonup b) = a,\\
        &(a\rightharpoonup b)\leftharpoonup (a\leftharpoonup b) = b,
    \end{align*}
    and the former actually implies the latter, by \eqref{braided-commutative}. Therefore, the involutivity of $\sigma$ is equivalent to $(a\rightharpoonup b)\rightharpoonup (a\leftharpoonup b) = a$, which, using the definition of $\sigma$, is rewritten as
    \[ (a\bullet b)^{-1} a (a\bullet b) = a  \]
for all $a\ot b\in \Gg\ot \Gg$. Now observe that, since $a\bullet \blank$ is bijective, this is the same as $c^{-1}a c = a$ for all $a,c$ with same source: i.e., $\cdot_\lambda$ is abelian for all $\lambda$. This concludes. 
\end{proof}

\subsection{Post-semiloopoids}\label{sec:post} Just as post-groupoids correspond to skew bracoids, we now give a notion of \textit{post-semiloopoid} that corresponds to quiver-theoretic skew braces. The name \textit{post-left unital associative semiloopoid} would be more fitting for this new structure, but we ditch it for being uselessly cumbersome, in favour of a shortened version. 

The notion of post-semiloopoid will appear quite natural: it is almost identical to the definition of post-groupoid, except that a technical hypothesis is dropped.
\begin{definition}
	A \textit{post-semiloopoid}\label{def:postsemiloopoid} $(p\colon \Gg\to \Lambda, \set{\cdot_\lambda}_{\lambda\in\Lambda}, \set{1_\lambda}_{\lambda\in\Lambda}, \Phi, \triangleright)$ is the datum of a group bundle $p\colon \Gg\to \Lambda$, where the group structures on the loops $\Gg_\lambda$ are denoted by $(\Gg_\lambda, \cdot_\lambda, 1_\lambda)$; of a map $\Phi\colon \Gg\to \Lambda$; and of a map $\triangleright\colon \Gg \fibre{\Phi}{p}\Gg\to \Gg$; satisfying the conditions (\textit{ii})--(\textit{vi}) from Definition \ref{def:postgroupoid}, and satisfying moreover $\Phi(1_\lambda) = \lambda$ for all $\lambda \in \im(\Phi)$.
\end{definition}
The only difference between Definition \ref{def:postsemiloopoid} and Definition \ref{def:postgroupoid} is that we do not require $\Phi$ to be surjective any more. When we try to repeat the construction of the Grossman--Larson groupoid, starting from a post-semiloopoid, we discover that not all vertices $\lambda\in\Lambda$ are targets, and hence $\Lambda$ is naturally divided into $\Lambda_{\mathrm{un}}= \im(\Phi)$ and $\Lambda_{\mathrm{in}} = \Lambda\smallsetminus\Lambda_{\mathrm{un}}$. Thus the following result is very much expected.
\begin{theorem}
	For a left unital associative semiloopoid $(\Gg,\bullet)$ over $\Lambda = \Lambda_{\mathrm{in}}\sqcup \Lambda_{\mathrm{un}}$, the following data are equivalent:
	\begin{enumerate}
		\item a quiver-theoretic skew brace structure on $\Gg$;
		\item a post-semiloopoid $(\source\colon \Gg\to \Lambda, \set{\cdot_\lambda}_{\lambda\in\Lambda}, \set{\mathbf{1}_\lambda}_{\lambda\in\Lambda}, \target, \triangleright)$ such that $x\bullet y = x\cdot_{\source(x)}(x\triangleright y)$ for all $x\ot y\in \Gg^{\ot 2}$.
	\end{enumerate}
\end{theorem}
\begin{proof}
	We first perform the construction from (\textit{i}) to (\textit{ii}). Given a quiver-theoretic skew brace, we define $p = \source$, $\Phi = \target$, on $p\colon \Gg\to \Lambda$ we put the group bundle structure given by the group operations $\cdot_\lambda$ of the quiver-theoretic skew brace, and finally we set $x\triangleright y = x\backslash_{\source(x)}(x\bullet y)$ for all $x\ot y \in \Gg^{\ot 2} = \Gg\fibre{\Phi}{p}\Gg$. By construction, $\Phi(1_\lambda) = \lambda$ if and only if $\lambda\in\Lambda_{\mathrm{un}}$. Since $\source(x\bullet y) = \source(x)$ and $\target(x\bullet y) = \target(y)$, one clearly has $p(x\triangleright y) = \source(x)$ and $\Phi(x\cdot_{\source(x)}(x\triangleright y)) = \Phi(x\bullet y) = \Phi(y)$ whenever these expressions are defined, thus proving conditions (\textit{ii}) and (\textit{iii}) from Definition \ref{def:postgroupoid}. Conditions (\textit{iv}) and (\textit{v}) follow from Proposition \ref{skew-brace-equivalent-conditions}. If $x\triangleright y = z$, then $y$ is retrieved as $y = x\backslash_{\source(x)} (z\bullet (x\triangleright y))$, and hence $x\triangleright\blank\colon \Gg(\target(x),\Lambda)\to \Gg(\source(x),\Lambda)$ is bijective, thus proving the condition (\textit{vi}) from Definition \ref{def:postgroupoid}.
	
	We now perform the converse construction from (\textit{ii}) to (\textit{i}). We set $\source = p$ and $\target = \Phi$. The group structures $\cdot_\lambda$ on $\Gg(\lambda,\Lambda)$ are defined as the group bundle operations $\cdot_\lambda$ of the post-semiloopoid, and the semiloopoid structure $\bullet $ is defined as in the Grossman--Larson groupoid, as $x\bullet y  = x\cdot_{p(x)}(x\triangleright y)$. The associativity of this operation follows exactly as in \cite[Theorem 2.16]{sheng2024postgroupoidsquivertheoreticalsolutionsyangbaxter}. For every $\lambda\in \im(\Phi)$, $y\in \Gg(\lambda,\Lambda)$ and $x\in \Gg(\Lambda, \lambda)$, one has
	\[ x\triangleright (1_\lambda\cdot_\lambda y) = (x\triangleright 1_\lambda)\cdot_{\source(x)} (x\triangleright y), \]
	and since $\cdot_\lambda$ is a group operation this implies $x\triangleright 1_\lambda = 1_{\source(x)}$. Thus $x\bullet 1_\lambda = x\cdot_{\source(x)} 1_{\source(x)} = x$, and this proves that $\Gg$ is unital. Finally, because $\Gg$ is a left unital associative semiloopoid and $\triangleright$ satisfies  Proposition \ref{skew-brace-equivalent-conditions} (\textit{iii}), we conclude that $\Gg$ is a quiver-theoretic skew brace.
\end{proof}

\section{Dynamical skew braces}\noindent \label{sec:dbraces}In this section, we recall the definition of dynamical sets, the dynamical Yang--Baxter equation, and dynamical skew braces. We shall observe that every dynamical skew brace is a quiver-theoeretic skew brace. In particular, zero-symmetric dynamical skew braces produce skew bracoids, and thus braided groupoids.
\subsection{Dynamical sets and the dynamical \textsc{ybe}} The dynamical Yang--Baxter equation (\textsc{dybe}), in its Lie-theoretic form, was introduced by Etingof and Schiffmann \cite{etingof2001lectures}. A purely set-theoretic version of the \textsc{dybe} was introduced by Shibukawa \cite{shibu}. Throughout this paper, ``\textsc{dybe}'' means Shibukawa's set-theoretic \textsc{dybe}.
\begin{definition}[Shibukawa \cite{shibu}]\label{dynamical_set}
A \emph{dynamical set} over a nonempty set $\Lambda$ is the datum of a set $X$, and of a map $\phi\colon \Lambda\times X\to \Lambda$ (sometimes called the \emph{transition map}).

Let $(X,\phi)$ and $(Y,\psi)$ be dynamical sets over $\Lambda$. A morphism $(X,\phi)\to (Y,\psi)$ of dynamical sets over $\Lambda$ is a family of maps $\set{f_\lambda\colon X\to Y}_{\lambda\in \Lambda}$, satisfying $\psi (\lambda, f_\lambda(x))= \phi(\lambda,x)$.
\end{definition}
With the above maps, dynamical sets over $\Lambda$ form a category, which we denote by $\mathsf{Set}_\Lambda$, following \cite{ashikaga2022dynamical}. The notation $\mathsf{DSet}_\Lambda$ is also used in literature \cite{matsumotoshimizu}. This is known to be a monoidal category, with the tensor product $\otimes$ of dynamical sets being defined as follows: $(X,\phi)\otimes (Y,\psi) $ is the set $X\times Y$ endowed with the transition map $\Lambda\times X\times Y\to \Lambda$, $(\lambda, x,y)\mapsto \psi(\phi(\lambda,x),y)$. A unit object is provided by a singleton $\set{*}$ with transition map $(\lambda,*)\mapsto \lambda$. For details about the associativity and unit constraints, we redirect the reader to \cite{matsumotoshimizu}.

\begin{remark}\hspace{-6pt}\footnote{This was pointed out by A.\@ Ardizzoni, to whom the author expresses his gratitude.}\hspace{6pt}\label{rem:slice} For a category $\Cc$ and an object $C\in\Obj(\Cc)$, the \emph{slice category} $\Cc/C$ is defined as the category whose objects are pairs $(X,f)$ where $X\in\Obj(\Cc)$ and $f$ is a morphism $X\to C$; and whose morphisms $(X,f)\to (Y,g)$ are the morphisms $h\colon X\to Y$ satisfying $f = gh$; see e.g.\@ \cite{riehl-context}.

Let $(X,\phi)$ be a dynamical set over $\Lambda$, and let $E:= \Lambda^\Lambda$ be the monoid of endofunctions of $\Lambda$; then, the transition map $\phi\colon \Lambda\times X\to \Lambda$ is uniquely identified by a map $\tilde{\phi}\colon X\to  E$, $x\mapsto \phi(\blank,x)$. There is clearly a bijective correspondence between the objects of $\mathsf{Set}_\Lambda$ and the objects of $\mathsf{Set}/E$. Every morphism $h\colon X\to Y$ in $\mathsf{Set}/E$ becomes a morphism in $\mathsf{Set}_\Lambda$ by $f_\lambda := h$ for all $\lambda$. 

It is known that every monoid structure $m\colon E\times E\to E$ on $E$ induces a monoidal structure $\ot_m$ on $\mathsf{Set}/E$ by
    \[ (X,\tilde{\phi})\ot_m (Y,\tilde{\psi}) := (X\times Y, m(\tilde{\phi}\times \tilde{\psi})).\]
    Choosing $m(f,g) := f\circ g$, one has that $\ot_m$ is the restriction of the monoidal structure $\ot$ of $\mathsf{Set}_\Lambda$. Thus $(\mathsf{Set}/E, \ot_m)$ is isomorphic to a wide (but non-full for $|\Lambda|>1$) monoidal subcategory of $(\mathsf{Set}_\Lambda, \ot)$.
\end{remark}
\begin{counterexample}
    Let $\Lambda = \set{\lambda,\mu}$, let $X = \set{x,y}$, let $\phi(\lambda,x) = \phi(\lambda,y) = \mu$, $\phi(\mu,x ) = \lambda$, $\phi(\mu, y) = \mu$, and consider the morphism of dynamical sets $f\colon (X,\phi)\to (X,\phi)$ defined by $f_\lambda(x) = y$, $f_\lambda(y) = x$, $f_\mu = \id_X$. It is clear that $f$ is not a morphism in $\mathsf{Set}/E$.
\end{counterexample}

We shall henceforth indicate a dynamical set $(X,\phi)$ by simply $X$.
\begin{definition}[Shibukawa \cite{shibu}]
    A \emph{solution to the} \textsc{dybe}, or \emph{dynamical Yang--Baxter map}, on a dynamical set $(X,\phi)$ over $\Lambda$ is a morphism $\sigma \colon X\ot X\to X\ot X$ of dynamical sets over $\Lambda$, satisfying the braid relation \eqref{eq:ybe}.
\end{definition}
An object $X$ of a monoidal category $(\Cc,\ot ,\One)$, endowed with a solution of the braid relation $\sigma\colon X\ot X\to X\ot X$, is usually called a \emph{braided object}, and $\sigma$ is called a \emph{braiding operator}. This nomenclature becomes somewhat problematic in the category of groups $\mathsf{Gp}$, and of groupoids $\mathsf{Gpd}_\Lambda$ over a set $\Lambda$. Indeed, in order to call a group(oid) ``braided'' it is generally required that the map $\sigma$ obeys the stronger conditions of Definition \ref{braided_groupoid}. For this reason, we shall rather speak of ``Yang--Baxter maps'' when we do not imply any stronger property.

Recall that the braided objects of a monoidal category $(\Cc,\ot,\One)$ form a subcategory $\mathsf{Br}(\Cc)$, whose morphisms are the morphisms $f\colon X\to Y$ in $\Cc$ such that $f\ot f$ intertwines the respective braidings $\sigma_X$ and $\sigma_Y$.
\begin{remark}
    When $\Lambda$ consists of a single element, the category $\mathsf{Br}(\mathsf{Set}_\Lambda)$ is isomorphic to the category $\mathsf{Br}(\mathsf{Set})$ of the solutions to the set-theoretic \textsc{ybe}.
\end{remark}
\begin{proposition}[{Matsumoto and Shimizu \cite{matsumotoshimizu}}] \label{monoidal_functor_DSet_to_Quiv}
There exists a fully faithful strong monoidal functor $\Quiv\colon \mathsf{Set}_\Lambda\to \mathsf{Quiv}_\Lambda$, given by sending $(X,\phi)$ to the quiver $\Quiv(X,\phi)=\Lambda\times X$, $\source(\lambda,x)=\lambda$, $\target(\lambda,x)=\phi(\lambda,x)$;
and sending the morphism $f =\set{f_\lambda\colon X\to Y}_{\lambda\in\Lambda}$ to the morphism $\Quiv(f)\colon \Lambda\times X\to \Lambda\times Y$, $(\lambda,x)\mapsto (\lambda, f_\lambda(x))$.

In particular, there is a fully faithful functor $\mathsf{Br}(\mathsf{Set}_\Lambda) \to \mathsf{Br}(\mathsf{Quiv}_\Lambda)$; and hence a dynamical Yang--Baxter map on a dynamical set can be interpreted as a Yang--Baxter map on a quiver.
\end{proposition}
In other words, every dynamical set can be seen as a quiver in the following way: given $(X,\phi)$, we construct the quiver $\Quiv(X,\phi)$ where each vertex $\lambda\in \Lambda$ has outgoing arrows labelled by $X$. The target of the arrow starting from $\lambda\in\Lambda$ and labelled by $x\in X$ is defined to be $\phi(\lambda,x)$. With this interpretation in mind, the following becomes straightforward:
\begin{lemma}[{Matsumoto and Shimizu \cite{matsumotoshimizu}}]\label{essential_image_of_Q}
A quiver $Q$ over $\Lambda$ lies in the essential image of $\Quiv$ if and only if the cardinality $n=|\source^{-1}(\lambda)|$ is independent of $\lambda\in \Lambda$. In this case, $Q$ is $\Quiv(X,\phi)$ for some dynamical set $(X,\phi)$ with $|X| = n$.
\end{lemma}
Notice that a complete quiver of degree $d$ always lies in the essential image of $\Quiv$.
\begin{notation}\label{notation:paths_on_quivers_of_dsets}
We introduce a notation for the arrows and paths in $\Quiv(X,
\phi)$. The unique arrow with source $\lambda$ and label $x$ will be denoted by $[\lambda \Vert x]$, and the path starting from $\lambda$ whose consecutive arrows are labelled $x_1,\dots, x_n$ will be denoted by $[\lambda\Vert  x_1|\dots| x_n]$.
\end{notation}
\subsection{Dynamical skew braces} In this section, we recall the definition of dynamical skew (left) braces: these algebraic structures are known to provide examples of dynamical Yang--Baxter maps. We shall prove, here, that their associated quivers are left unital associative semiloopoids (which will turn out to be braided), and they are moreover groupoids if the dynamical skew brace is zero-symmetric.
\begin{definition}[Anai \cite{anai2021thesis}, generalising Matsumoto \cite{matsu}]\label{dynamical_skew_brace}
    A \emph{dynamical skew} (\emph{left}) \emph{brace} $(A,\Lambda, \phi, \cdot, \set{\bullet_\lambda}_{\lambda\in\Lambda})$ is the datum of a dynamical set $(A,\phi)$ over $\Lambda$, a group structure $(A,\cdot)$ on $A$, and a family of left quasigroup structures $\set{(A,\bullet_\lambda)}_{\lambda\in\Lambda}$ on $A$, such that the following conditions hold for all $\lambda\in\Lambda$, $a,b,c\in A$:
		\begin{align}&\label{eq:dyn-associativity}\tag{\sc da}a\bullet_{\lambda}(b\bullet_{\phi(\lambda,a)} c)=(a\bullet_{\lambda} b)\bullet_{\lambda} c,\\
		\label{eq:brace-compatibility}\tag{\sc bc}&a\bullet_\lambda(b\cdot c) = (a\bullet_\lambda b)\cdot a^{-1}\cdot( a\bullet_\lambda c),\end{align}
		where $a^{-1}$ denotes the inverse with respect to $\cdot$. We denote $(A,\Lambda, \phi, \cdot, \set{\bullet_\lambda}_{\lambda\in\Lambda})$ simply by $A$, when the additional structures are understood.

  A \emph{morphism} $A \to B$ of dynamical skew braces over $\Lambda$ is a morphism $f=\set{f_\lambda}_{\lambda\in\Lambda}\colon (A,\phi_A)\to (B,\phi_B)$ of dynamical sets, such that $f_\lambda(a\bullet_\lambda b) = f_\lambda(a)\bullet_\lambda f_{\phi(\lambda,a)}(b)$, $f_\lambda(\mathbf{1}_\lambda)= \mathbf{1}_\lambda$ and $f_\lambda(a\cdot b) = f_\lambda(a)\cdot f_\lambda(b)$ for all $\lambda\in\Lambda$. We denote by $\mathsf{DSB}_\Lambda$ the category of dynamical skew braces over $\Lambda$.

A \emph{dynamical brace}, or a dynamical skew brace \emph{of abelian type}, is a dynamical skew brace $A$ with $(A,\cdot)$ abelian.
\end{definition}
Condition \eqref{eq:dyn-associativity} is also called the \emph{dynamical associativity}, and \eqref{eq:brace-compatibility} is the usual \emph{brace compatibility}. For the rest of this section, let $\backslash_\lambda$ be the left division with respect to $\bullet_\lambda$.
\begin{remark}\label{Matsumoto_gave_a_different_definition}
    The definition of a dynamical brace given by Matsumoto is slightly different from ours. However, the two definitions are equivalent by \cite[Proposition 3.3]{matsu}.
\end{remark}
\begin{remark}\label{computation_rules_for_braces}
Since $\bullet_\lambda$ satisfies the brace compatibility \eqref{eq:brace-compatibility}, the following computation rules are known to hold (see e.g.\@ the survey \cite{vendramin2023skew}):
\begin{equation}\label{eq:brace-computation-rules}
a\bullet_\lambda (b^{-1}\cdot c)= a\cdot (a\bullet_{\lambda} b)^{-1}\cdot (a\bullet_{\lambda} c),\quad a\bullet_{\lambda}(b\cdot c^{-1})=(a\bullet_{\lambda} b)\cdot (a\bullet_{\lambda} c)^{-1}\cdot a.
\end{equation}
\end{remark}
We shall henceforth drop the $\cdot$ from the notation.

In the case of a dynamical brace, the Matsumoto--Shimizu functor $\Quiv$ takes an explicit form, involving the group $A$ and its \emph{holomorph} $\Hol(A)$. We recall that the holomorph of a group is defined as $\Hol(A) := A\rtimes \Aut(A)$, where the action of $\Aut(A)$ on $A$ is the evaluation $\Aut(A)\times A\to A$, $(f, a)\mapsto f(a)$.
\begin{definition}
    Let $(A,\cdot)$ be a group which is also a dynamical set $(A, \phi)$ over $\Lambda$. We say that $A$ is a \textit{dynamical group of bundle type} if the \textit{action rule}
    \begin{equation}\label{action_rule}\tag{\sc ar}
   	\phi(\lambda, ab) = \phi( \phi(\lambda, a), b)
    \end{equation}
    is satisfied for all $a,b\in A$, $\lambda\in\Lambda$. 
\end{definition}
\begin{remark}
The Matsumoto--Shimizu functor clearly takes dynamical groups of bundle type into bundles of groups, where each group of loops is isomorphic to $(A,\cdot)$. Indeed, \eqref{action_rule} is equivalent to the request that the group multiplication $m\colon A\times A\to A$ induces a morphism of dynamical sets $\set{m_\lambda = m}_{\lambda\in\Lambda}\colon A\ot A\to A$. \end{remark}
Observe that if $(A,\cdot)$ is a group and $(A,\phi)$ is a dynamical set over $\Lambda$, then $\Hol(A)$ inherits the structure of a dynamical set by simply disregarding the second entry:
\[ \phi_{\Hol(A)}(\lambda, (a,f)):= \phi(\lambda,a). \]
Clearly, $\Hol(A)$ is a dynamical group of bundle type if and only if $A$ is. 
\begin{definition}[{see \cite[\S4]{matsu}}]\label{dynamical_subgroup}
	Let $A$ be a dynamical group of bundle type. A \emph{dynamical subgroup} of $\Hol(A)$ is a family $\Ss=\set{S_\lambda}_{\lambda\in\Lambda}$ of subsets $\emptyset\neq S_\lambda\subseteq \Hol(A)$ such that, for all $(a,f)\in S_\lambda$, one has $(a,f)^{-1}S_\lambda =  S_{\phi(\lambda, a)}$.
	
		 Notice that this implies $\phi(\lambda,1_A)=\lambda$. Observe that a dynamical subgroup of $\Hol(A)$ is \emph{not} a subgroup of $\Hol(A)$, unless $|\Lambda| = 1$. 
		 
		 We say that the dynamical subgroup is \emph{unital} if $1_A\in S_\lambda$ for all $\lambda$.
	
	We say that a subset $S\subseteq \Hol(A)$ is \emph{regular} if the projection on $A$ restricts to a bijection $S\to A$. A dynamical subgroup $\Ss = \set{S_\lambda}_{\lambda\in\Lambda}$ of $\Hol(A)$ is said to be \emph{regular} if $S_\lambda$ is a regular subset of $\Hol(A)$ for all $\lambda\in \Lambda$.
\end{definition}
\begin{proposition}[Matsumoto \cite{matsu}, Anai \cite{anai2021thesis}]\label{dynamical_braces_and_dynamical_subgroups}
Let $(A,\phi)$ be a dynamical set over $\Lambda$ which is also a group. There is a bijective correspondence between:
\begin{enumerate}
    \item families of left quasigroup structures $\set{\bullet_\lambda}_{\lambda\in\Lambda}$ such that $(A,\Lambda,\phi,\cdot,\set{\bullet_\lambda}_{\lambda\in\Lambda})$ is a dynamical skew brace;
    \item regular dynamical subgroups $\Ss$ of $\Hol(A)$.
\end{enumerate}
The correspondence is as follows:
\begin{enumerate}
    \item[(\textit{i}$\,$$\rightarrow$$\,$\textit{ii})] Given a family $\set{\bullet_\lambda}_{\lambda\in\Lambda}$, set $f^\lambda_a\colon b\mapsto a^{-1}(a\bullet_\lambda b)$, and $\Ss=\set{S_\lambda}_{\lambda\in\Lambda}$ with $S_\lambda:=\set{(a, f^\lambda_a)\mid a\in A}$.
    \item[(\textit{ii}$\,$$\rightarrow$$\,$\textit{i})] Given $\Ss=\set{S_\lambda}_{\lambda\in\Lambda}$ regular, for all $\lambda\in \Lambda$ and for all $a\in A$ there exists a unique automorphism $f^\lambda_a$ such that $(a,f^\lambda_a)\in S_\lambda$. Then, set $a\bullet_\lambda b := af^\lambda_a(b) $.
\end{enumerate}
The brace compatibility on $\bullet_\lambda$ is equivalent to $f^\lambda_a$ being a group automorphism of $(A,\cdot)$. The dynamical associativity on the family $\set{\bullet_\lambda}_{\lambda\in\Lambda}$ is equivalent to the family $\Ss$ being a dynamical subgroup, which is regular by definition.
\end{proposition}
    We shall equivalently say that $(A,\Lambda,\phi,\cdot,\set{\bullet_\lambda}_{\lambda\in\Lambda})$ is a dynamical skew brace, or that $(A,\cdot,\Ss)$ is a dynamical skew brace, according to which definition comes most in handy.

    For a dynamical skew brace $(A,\cdot, \Ss)$, we can construct a quiver $Q$ in the following way. The set of vertices of $Q$ is $\Lambda$. Notice that for each $\lambda\in \Lambda$ and $a\in A$ one has $(a,f^\lambda_a)^{-1}S_\lambda = S_{\phi(\lambda,a)}$, thus we associate with each $a\in A$ an arrow in $Q(\lambda,\phi(\lambda,a))$. Observe that $Q $ is exactly $ \Quiv(A,\phi)$. 
\begin{remark}[{cf.\@ \cite[Proposition 5.1]{matsumoto2011combinatorialaspects}}]\label{looped_vertex_equivalent_conditions}
    Let $Q =\Quiv(A,\phi)$. Notice that the request $\beta^{-1}\alpha \in S_{\phi_{\Hol(A)}(\lambda,\beta)}$ for all $\alpha,\beta\in S_\lambda$ implies $1_A\in S_{\phi_{\Hol(A)}(\lambda,\beta)}$. Therefore, $1_A\in S_\lambda$ for all $\lambda$ such that $Q(\Lambda,\lambda)\neq \emptyset$. The converse also holds: if $1_A\in S_\lambda$, then $1_A^{-1}S_\lambda = S_\lambda$, and this is an incoming arrow at $\lambda$ (it is, in fact, a loop on $\lambda$). Thus, for a dynamical subgroup $\Ss$ of $\Hol(A)$, and for a vertex $\lambda\in\Lambda$, the following are equivalent:
    \begin{enumerate}
    \item $1_A\in S_\lambda$;
        \item the vertex $\lambda$ has at least one incoming arrow;
        \item the vertex $\lambda$ has a loop.
    \end{enumerate}
    As an immediate consequence, the dynamical subgroup $\Ss$ is unital if and only if all vertices $\lambda\in\Lambda$ satisfy the above equivalent conditions.
\end{remark}
\begin{lemma}\label{lemma:inverses_with_f} Let $(A,\Lambda,\phi,\cdot,\set{\bullet_\lambda}_{\lambda\in\Lambda})$ be a dynamical skew brace, and $f^\lambda_a\colon b\mapsto a^{-1}(a\bullet_\lambda b)$ be defined as above. Then, one has \[f^{\phi(\lambda,a)}_{\left(\inv{f^\lambda_a}(a)\right)^{-1}} = \inv{f^\lambda_a}.\] 
\end{lemma}
\begin{proof}We use the description given in Proposition \ref{dynamical_braces_and_dynamical_subgroups}. Since the pair \[\left(\left(\inv{f^\lambda_a}(a)\right)^{-1}\!\!,\; \inv{f^\lambda_a}\right)\] lies in the subset $S_{\phi(\lambda,a)}$, by definition of $f^{\phi(\lambda,a)}$ one has $f^{\phi(\lambda,a)}_{\left(\inv{f^\lambda_a}(a)\right)^{-1}} = \inv{f^\lambda_a}$. \end{proof}
\begin{remark}\label{quiver_of_groupoid_is_homogeneous}Let $(\Gg,\bullet)$ be a groupoid over $\Lambda$. It is well known that, as a quiver, $\Gg$ is a disjoint union of  complete quivers $\Kk_i$ of some degrees $d_i$, possibly distinct. Indeed, let $\lambda,\mu,\nu\in \Lambda$ be such that there exists an arrow $a\colon \mu \to \nu$: then, $\blank\bullet a$ is a bijection $\Gg(\lambda,\mu)\to \Gg(\lambda,\nu)$ with inverse $\blank\bullet a^{-\bullet}$, where $a^{-\bullet}$ denotes the inverse of $a$ with respect to $\bullet$. The dual argument yields $\Gg(\mu,\lambda)\cong \Gg(\nu,\lambda)$.
\end{remark}
\begin{theorem}\label{groupoid_structure_on_Q(A)}
    Let $(A,\cdot,\Ss)$ be a dynamical skew brace. The following is a left unital associative semiloopoid structure on $Q$:
    \begin{enumerate}
        \item the composition $\bullet$ of the arrow $[\lambda\Vert a]$ and the arrow $[\phi(\lambda ,a)\Vert b]$ is the arrow $[\lambda \Vert af^\lambda_a(b)] =[\lambda\Vert a\bullet_\lambda b]$;
        \item the set $\Ss$ is decomposed into $\Ss_\mathrm{in}\sqcup\Ss_{\mathrm{un}}$, where by definition $S_\lambda\in \Ss_{\mathrm{un}}$ if and only if $(1_A, \id_A)\in S_\lambda$, and the decomposition $\Lambda = \Lambda_{\mathrm{in}}\sqcup \Lambda_{\mathrm{un}}$ is defined accordingly; \item the right unit $\mathbf{1}_\lambda$ on the vertex $\lambda\in \Lambda_{\mathrm{un}}$ is $[\lambda\Vert 1_A]$, which is a loop.
        \end{enumerate}
        The full subquiver with vertices $\Lambda_{\mathrm{un}}$ is a groupoid, it is clearly the maximal left sub-semiloopoid of $Q$ which is a groupoid. Moreover, $Q$ itself is a groupoid if and only if $\Ss$ is unital (i.e., if and only if $A$ is zero-symmetric). For a unital vertex $\lambda\in\Lambda_{\mathrm{un}}$, one has:
        \begin{enumerate}\setcounter{enumi}{3}
        \item the inverse of $[\lambda\Vert a]$, denoted by $[\lambda\Vert a]^{-\bullet}$, is \[ [\lambda\Vert a]^{-\bullet} = \Big[\phi(\lambda,a)\Big\Vert \left(a\backslash_\lambda (a\cdot a)\right)^{-1} \Big] = \bigg[ \phi(\lambda,a)\bigg\Vert \left(\inv{f^\lambda_a}(a)\right)^{-1} \bigg],\] where $\inv{f^\lambda_a}$ denotes the inverse of the map $f^\lambda_a$, and $\left(\inv{f^\lambda_a}(a)\right)^{-1}$ denotes the inverse in $A$ of the element $\inv{f^\lambda_a}(a)$.
    \end{enumerate}
\end{theorem}
Observe, in  particular, that the multiplication $\bullet$ in the left unital associative semiloopoid $Q$ is uniquely described by the family $\set{\bullet_\lambda}_{\lambda\in\Lambda}$.
\begin{proof}
In this proof, $(\blank)^{-1}$ denotes either the inverse in $(A,\cdot)$, or the inverse in $\Hol(A)$, or the inverse function of an invertible map $f\colon A\to A$, according to the situation. There will always be a unique sensible interpretation for it, but the reader is invited to use some care.

    Let $[\lambda\Vert a]\bullet [\phi(\lambda,a)\Vert  b] := [\lambda\Vert  a f^\lambda_a(b)]$. In order to prove that the left semiloopoid structure $\bullet$ is well defined, we need to check that $a f^\lambda_a(b)$ labels indeed an edge $\lambda\to \phi(\phi(\lambda,a),b)$. Notice that the statement: «$(b,f^{\phi(\lambda,a)}_b)\in S_{\phi(\lambda,a)}$» is equivalent to: «There exists $(x, f^\lambda_x)\in S_\lambda$ such that $(b,f^{\phi(\lambda,a)}_b) = (a,f^\lambda_a)^{-1}(x, f^\lambda_x)$». Thus $ S_{\phi(\phi(\lambda,a),b)} = (x, f^\lambda_x)^{-1}(a,f^\lambda_a)(a, f^\lambda_a)^{-1}S_\lambda = (x, f^\lambda_x)^{-1}S_\lambda$, and now it suffices to observe that $x = af^\lambda_a(b)$ does the job.
   
As for the associativity of $\bullet$, notice that
    $$\Big([\lambda\Vert a]\bullet [\phi(\lambda,a)\Vert b]\Big)\bullet [\phi(\phi(\lambda,a),b)\Vert c]=[\lambda\Vert  (a\bullet_\lambda b)\bullet_\lambda c],$$
    and 
    $$[\lambda\Vert a]\bullet\Big([\phi(\lambda,a)\Vert b]\bullet [\phi(\phi(\lambda,a),b)\Vert c] \Big) = [\lambda\Vert  a\bullet_{\lambda}(b\bullet_{\phi(\lambda,a)} c)],$$
    thus the associativity of $\bullet$ follows from the dynamical associativity of $\set{\bullet_\lambda}_{\lambda\in\Lambda}$.

One has $a\bullet_\lambda 1_A = af^\lambda_a(1_A)=a$. If moreover $\lambda\in\Lambda_{\mathrm{un}}$, i.e.\@ if $(1_A,\id_A)\in S_\lambda$, one also has $f^\lambda_{1_A} = \id_A$, whence $1_A\bullet_\lambda a = 1_A f^\lambda_{1_A}(a)= a$. Observe that $[\lambda\Vert 1_A] $ is a loop on $\lambda$ if $\lambda\in \Lambda_{\mathrm{un}}$, because $\Ss_{\phi(\lambda,1_A)} = (1_A,\id_A)^{-1}\Ss_\lambda = \Ss_\lambda$. This proves that $[\lambda\Vert 1_A]$ is a loop and a bilateral unit if $\lambda\in\Lambda_{\mathrm{un}}$. 

Suppose now $(1_A,\id_A)\in S_\lambda$. If $(a,f^\lambda_a)$ is an element of $S_\lambda$ which produces an arrow $[\lambda\Vert  a]\colon \lambda\to \phi(\lambda,a)$, then $(a,f^\lambda_a)^{-1}(1_A,\id_A) = (a, f^\lambda_a)^{-1}$ is an element of $S_{\phi(\lambda,a)}$, which corresponds to an arrow \[\bigg[\phi(\lambda,a)\bigg\Vert  \left(\inv{f^\lambda_a}(a)\right)^{-1}\bigg]\!\colon \phi(\lambda,a)\to \lambda.\] One has
$a\bullet_\lambda \left(\left(\inv{f^\lambda_a}(a)\right)^{-1}\right)= a\left(f^\lambda_a\inv{f^\lambda_a}(a)\right)^{-1}=1_A$, thus \[  [\lambda\Vert a]\bullet \Big[\phi(\lambda,a)\Big\Vert  \inv{\left(\inv{f^\lambda_a}(a)\right)}\Big]=[\lambda\Vert 1_A].\] By Lemma \ref{lemma:inverses_with_f}, one has $f^{\phi(\lambda,a)}_{\left(\inv{f^\lambda_a}(a)\right)^{-1}} = \inv{f^\lambda_a}$, whence
\begin{align*}
\left(\left(\inv{f^\lambda_a}(a)\right)^{-1} \right)\bullet_{\phi(\lambda,a)}a&= \left(\inv{f^\lambda_a}(a)\right)^{-1}f^{\phi(\lambda,a)}_{\left(\inv{f^\lambda_a}(a)\right)^{-1}}(a)\\
&=\left(\inv{f^\lambda_a}(a)\right)^{-1}\inv{f^\lambda_a}(a)\\&=1_A,
\end{align*}
thus \[ \Big[\phi(\lambda,a)\Big\Vert \left(\inv{f^\lambda_a}(a)\right)^{-1}\Big]\bullet [\lambda\Vert a]=[\phi(\lambda,a)\Vert 1_A].\] This proves that $\Big[\phi(\lambda,a)\Big\Vert \left(\inv{f^\lambda_a}(a) \right)^{-1}\Big]$ is the bilateral inverse of $[\lambda\Vert a]$, and hence the full subquiver with vertices $\Lambda_{\mathrm{un}}$ is a groupoid. In particular, $Q$ is a groupoid if and only if $\Lambda = \Lambda_{\mathrm{un}}$. Furthermore, notice that $f^{\lambda}_a(c)= a^{-1}(a\bullet_\lambda c)= b$ implies $c = \inv{f^\lambda_a}(b) = a\backslash_\lambda (a\cdot b)$, whence also \[ [\lambda\Vert a]^{-\bullet} = \bigg[\phi(\lambda,a)\bigg\Vert \left(\inv{f^\lambda_a}(a) \right)^{-1}\bigg] = \Big[\phi(\lambda,a)\Big\Vert \left(a\backslash_\lambda (a\cdot a)\right)^{-1} \Big].\]
Finally, let $\lambda$ be any vertex: since $[\lambda\Vert a]\bullet [\phi(\lambda,a)\Vert b] = [\lambda\Vert a\bullet_\lambda b]$ holds, one deduces that $[\lambda\Vert a]\bullet \blank$ is a bijection $Q(\phi(\lambda,a),\Lambda)\to Q(\lambda,\Lambda)$ because $a \bullet_\lambda\blank$ is a bijection $A\to A$.
\end{proof}
\begin{corollary}\label{cor:Q(A)_homogeneous}
    Let $(A,\cdot,\Ss)$ be a dynamical skew brace. Then $\Ss$ is unital if and only if $Q=\Quiv(A,\phi)$ is a homogeneous quiver of weight $|A|$.
\end{corollary}
\begin{proof}
    If $Q$ is a union of complete quivers, then in particular every vertex has a loop, hence $\Ss$ is unital by Remark \ref{looped_vertex_equivalent_conditions}. Conversely, suppose that $\Ss$ is unital. Then $Q$ is the underlying quiver of a groupoid, thus by Remark \ref{quiver_of_groupoid_is_homogeneous} $Q$ is a union of complete quivers of some degrees $d_i$. Let $\Kk_i$ be a connected component of $Q$, complete of degree $d_i$: if $\lambda$ is a vertex in $\Kk_i$ then, for all vertices $\mu$ in $\Kk_i$, there exist exactly $d_i$ arrows from $\lambda$ to $\mu$. Moreover, the number of arrows with source $\lambda$ is $|Q(\lambda,\Lambda)| =|A|$. Thus $|A| = d_i |\Obj(\Kk_i)|$ holds.
\end{proof}
\begin{remark}
Following \cite[Definition 3.4]{matsu}, a multiplication $\bullet_\lambda$ is \emph{zero-symmetric} if $1_A\bullet_\lambda a= a$ holds for all $a$; and a dynamical brace is \emph{zero-symmetric} if all the $\bullet_\lambda$'s are zero-symmetric. In Theorem \ref{groupoid_structure_on_Q(A)}, we have proven \textit{en passant} that $A$ is zero-symmetric if and only if $\Ss$ is unital, if and only if $\Quiv(A,\phi)$ has a groupoid structure; cf.\@ \cite[Proposition 5.1]{matsumoto2011combinatorialaspects}. We denote by $\mathsf{DSB}^0_\Lambda\subset \mathsf{DSB}_\Lambda$ the subcategory of zero-symmetric dynamical skew braces.
\end{remark}
\begin{remark}
Skew braces are always zero-symmetric.
\end{remark}
\begin{remark}\label{zero_symmetric_string_of_integers}If a dynamical skew brace is zero-symmetric, then $\Quiv(A,\phi)$ is homogeneous of weight $|A|$ by Corollary \ref{cor:Q(A)_homogeneous}, and hence its shape is uniquely determined by a string of integers $\set{N_s}_s$, where $N_s$ is the number of connected components with $s$ vertices, and $s$ runs along the divisors of $|A|$.\end{remark}
\subsection{Dynamical skew braces and braidings} Dynamical skew braces have been purposely introduced to produce solutions to the \textsc{ybe}. Thus, the following proposition is central in our work.
\begin{proposition}[Anai \cite{anai2021thesis}, generalising Matsumoto \cite{matsu}]\label{braiding-on-dbraces}For every dynamical skew brace $(A,\Lambda, \phi, \cdot, \set{\bullet_\lambda}_{\lambda\in\Lambda})$, the following is a Yang--Baxter map on the left semiloopoid $\Quiv(A,\phi)$:
    \begin{equation}\label{eq:dbrace-braiding}\begin{split}&\sigma\Big([\lambda\Vert a]\ot  [\phi(\lambda,a)\Vert b]\Big)\\ &\hspace{2em}= \big[\lambda\big\Vert  a^{-1}(a\bullet_\lambda b)\big]\ot  \big[ \phi(\lambda, a^{-1}(a\bullet_\lambda b))\big\Vert  (a^{-1}(a\bullet_\lambda b))\backslash_{\lambda} (a\bullet_\lambda b) \big] \\
    &\hspace{2em}= \big[\lambda\big\Vert  f^\lambda_a(b)\big]\ot  \big[ \phi(\lambda, f^\lambda_a(b))\big\Vert f^\lambda_a(b)\backslash_{\lambda} \big( af^\lambda_a(b)\big) \big] .\end{split}\end{equation}
The braiding is involutive if and only if $A$ is a dynamical brace.
\end{proposition}
\begin{zusatz}[to Proposition \ref{braiding-on-dbraces}] \label{zusatz}
    Let $(A,\Lambda, \phi, \cdot, \set{\bullet_\lambda}_{\lambda\in\Lambda})$ be a dynamical skew brace. Then the left unital associative semiloopoid $Q=\Quiv(A,\phi)$ has the structure of a quiver-theoretic skew brace, and the braiding $\sigma$ on $Q$ is exactly the map \eqref{eq:dbrace-braiding}. If moreover $\Ss$ is unital (i.e., $A$ is zero-symmetric), then $Q$ is a skew bracoid, and $\sigma$ is the corresponding non-degenerate braiding on the groupoid $Q$.
\end{zusatz}
\begin{proof}[Proof of the Addendum]
    We denote the two components of $\sigma$ by $\sigma(a\ot b) :=(a\rightharpoonup b)\ot( a\leftharpoonup b) $, as usual. Let $Q =\Quiv(A,\phi)$: we denote by $m\colon Q\otimes Q\to Q$, $m(x\ot y)=x\bullet y$, its groupoid multiplication. 
    
    The definition of $\sigma$ immediately implies $a\bullet_{\lambda} b = (a\rightharpoonup b)\bullet_\lambda (a\leftharpoonup b)$ for all $a,b\in A$, $\lambda\in \Lambda$. This proves $m\sigma= m$, which is \eqref{braided-commutative}.

We now define the quiver-theoretic skew brace structure on $Q$. The group operations $\cdot_\lambda$ on each vertex are simply given by $[\lambda\Vert a]\cdot_\lambda [\lambda\Vert b]:= [\lambda\Vert a\cdot b]$, and they satisfy $x\rightharpoonup y = x^{-1}\cdot (x\bullet y)$. The compatibility \eqref{qtbc} follows immediately from \eqref{eq:brace-compatibility}. This is moreover a skew bracoid when $Q$ is a groupoid, i.e, when $\Lambda = \Lambda_{\mathrm{un}}$, i.e., when $A$ is zero-symmetric. Observe that $\sigma$ is exactly the braiding defined in Theorem \ref{thm:qtsb_yields_braiding}.
\end{proof}
We have recalled the definition of a dynamical skew \emph{left} brace. One may define a notion of dynamical skew \emph{right} brace, by replacing left quasigroups with right quasigroups, and by assuming a right brace compatibility. A right version of the previous results, then, compiles itself: dynamical skew \emph{right} braces produce braided \emph{right} unital associative semiloopoids, and hence Yang--Baxter maps; and they also produce braided groupoids if they are zero-symmetric.

\subsection{The maximal (zero-symmetric) dynamical skew brace of a group} In this section, we shall present a class of examples of (zero-symmetric) dynamical skew braces that contains all other possible (zero-symmetric) examples. In particular, this will provide concrete examples of skew bracoids (see Examples \ref{ex:Z3} and \ref{ex:Z4}), and of quiver-theoretic skew braces that are not skew bracoids (see Example \ref{ex:Z3_initials}).

 Let $(A,\cdot)$ be any group, and let $\Ss_A$ be the set of all regular subsets of $\Hol(A)$. This is obviously a regular dynamical subgroup, it is the maximal regular dynamical subgroup of $\Hol(A)$, and yields a maximal dynamical skew brace over $(A,\cdot)$. Every other dynamical skew brace over $(A,\cdot)$ is contained in the maximal one.

Similarly, there is a maximal \emph{unital} regular dynamical subgroup of $\Hol(A)$, namely the set $\Ss^0_A = \set{S\subset \Hol(A) \text{ regular }, (1_A, \id_A)\in S}$. Indeed, the set $S_0 := \set{(a, \id_A)\mid a\in A}$ is in $\Ss^0_A$; and, if $S$ lies in $ \Ss^0_A$ and $(a , f)\in S$, then $(a, f)^{-1}S$ contains $(a,f)^{-1}(a,f) = (1_A, \id_A)$, and hence $(a,f)^{-1}S$ lies in $ \Ss^0_A$. It is clear by construction that $\Ss^0_A$ is the maximal \emph{unital} regular dynamical subgroup of $\Hol(A)$. The corresponding dynamical skew brace contains all the other \emph{zero-symmetric} dynamical skew braces over $(A,\cdot)$.
\begin{remark}
    It is easy to get convinced that $|\Ss^0_A| = |\Aut(A)|^{|A|-1}$: indeed, for all $a\in A$ we need to choose a map $f^\lambda_a$ among the automorphisms of $A$, but for $a=0$ the choice $f^\lambda_0 = \id_A$ is forced. Similarly, one can see that $|\Ss_A| = |\Aut(A)|^{|A|}$.
\end{remark}
\begin{example}\label{ex:Z3}
     Let $A=\Z/3\Z$, thus $\Aut(A)= \set{\pm \id}$. The maximal zero-symmetric dynamical brace on $A$ is $\Ss^0_A = \set{S_0, S_1, S_2, S_3}$, where
    	\begin{align*}
		&S_0 := \set{(0,\id), (1,\id), (2,\id)}, & & S_1 := \set{(0,\id), (1,-\id), (2,\id)},\\
		&S_2 := \set{(0,\id), (1,\id), (2,-\id)}, && S_3 := \set{(0,\id), (1,-\id), (2,-\id)},
	\end{align*}
 and the resulting quiver is as follows:
 $$
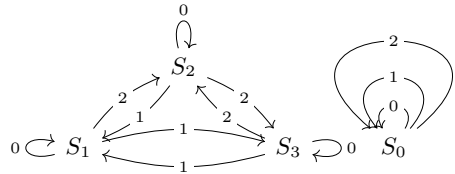
\begin{tikzcd}
& S_2\ar[ld,bend left=15, "1"description]\ar[rd,bend left=15, "2"description] \ar[loop above, "0"]& &\\
S_1\ar[ru,bend left=15, "2"description]\ar[rr,bend left=15, "1"description]\ar[loop left, "0"]& &S_3\ar[ll,bend left=15,"1"description]\ar[lu, bend left=15, "2"description]\ar[loop right, "0"]& S_0\ar[loop, out=60, in=120, looseness=4, "0"description]\ar[loop, out=50, in=130, looseness=7, "1"description]\ar[loop, out=40, in=140, looseness=10, "2"description]
\end{tikzcd}$$
We call $\Kk_1$ the connected component with three vertices, and $\Kk_0$ the connected component with one vertex.

We now describe the associated Yang--Baxter map. Notice that every path of the form $[S_i\Vert a|0]$ is sent into $[S_i\Vert 0|a]$, and vice versa. One may verify by hand that all the other paths of length $2$ in $\Kk_1$ are fixed by the solution. As for the component $\Kk_0$, it is easy to see that $[S_0\Vert  2|1]$ maps to $[S_0\Vert 1|2]$ and vice versa, $[S_0\Vert  a|0]$ maps to $[S_0\Vert  0|a]$ and vice versa, and all the other paths are fixed. Observe that this is, for groupoids, the closest we can get to a ``canonical flip''. The family of left quasigroup operations $\set{\bullet_\lambda}_\lambda$, for the vertices $\lambda$ in $\Kk_1$, is described in Table \ref{table:bullet_isolated}; while $a\bullet_{S_0} b = a+b$. See also \cite[Example 5.2]{matsumoto2011combinatorialaspects}, \cite[Example 5.3]{matsu}.
\begin{table}[t]
    \centering
    \begin{tabular}{c|ccc}
         $\bullet_{S_1}$& 0 & 1 & 2  \\ \hline
        0 &0& 1& 2 \\
        1&1& 0& 2\\
        2 &2& 0 & 1
    \end{tabular}\qquad \begin{tabular}{c|ccc}
         $\bullet_{S_2}$& 0 & 1 & 2  \\ \hline
        0 &0& 1& 2 \\
        1&1& 2& 0\\
        2 &2& 1 & 0
    \end{tabular}\qquad 
    \begin{tabular}{c|ccc}
         $\bullet_{S_3}$& 0 & 1 & 2  \\ \hline
        0 &0& 1& 2 \\
        1&1& 0& 2\\
        2 &2& 1 & 0
    \end{tabular}
    \caption{The family of left quasigroup operations on $\Kk_1$.}
    \label{table:bullet_isolated}
\end{table}
\end{example}
\begin{example}\label{ex:Z3_initials}
    The maximal dynamical skew brace of $A = \Z/3\Z$ is obtained by attaching initial vertices to the quiver of Example \ref{ex:Z3}. The initial vertices in $\Ss_A$ are
    \begin{align*}
		&R_0 := \set{(0,-\id), (1,-\id), (2,-\id)}, & & R_1 := \set{(0,-\id), (1,-\id), (2,\id)},\\
		&R_2 := \set{(0,-\id), (1,\id), (2,-\id)}, && R_3 := \set{(0,-\id), (1,\id), (2,\id)},
	\end{align*}
 and the resulting quiver is as follows (the initial vertices are in gray):
 $$
\begin{tikzcd}
\textcolor{gray}{R_3}\ar[rrr, color=gray, "1"description]\ar[rrd, color=gray, "2"description]\ar[rrrrd, color=gray, "0"description]& & & S_2\ar[ld,bend left=15, "1"description]\ar[rd,bend left=15, "2"description] \ar[loop above, "0"]& & &&\\
&&S_1\ar[ru,bend left=15, "2"description]\ar[rr,bend left=15, "1"description]\ar[loop left, "0"]& &S_3\ar[ll,bend left=15,"1"description]\ar[lu, bend left=15, "2"description]\ar[loop right, "0"]& S_0\ar[loop, out=60, in=120, looseness=4, "0"description]\ar[loop, out=50, in=130, looseness=7, "1"description]\ar[loop, out=40, in=140, looseness=10, "2"description]&&\textcolor{gray}{R_0}\ar[ll, color=gray, "1"description]\ar[ll, bend left=20, color=gray, "0"description]\ar[ll, bend right=20, color=gray, "2"description]\\
\textcolor{gray}{R_2}\ar[rru, color=gray, "2"description]\ar[rrruu, color=gray, bend left=20, "0"description]\ar[rrrru, color=gray, "1"description] & &\textcolor{gray}{R_1}\ar[u, color=gray, "0"description]\ar[ruu, color=gray, bend right=10,"1"description]\ar[rru, color=gray, "2"description] &  &  &  & &
\end{tikzcd}$$
\end{example}
\begin{example}\label{ex:Z4}
     Let $A=\Z/4\Z$, $\Aut(A)=\set{\pm \id}$. The following is the family $\Ss^0_A$:
    	\begin{align*}
	S_0 =& \set{(0,\id),\; (1,\id),\; (2,\id),\; (3,\id)}; & S_2 =& \set{(0,\id),\; (1,\id),\; (2,-\id),\; (3,-\id)}; \\ 
	S_4 =& \set{(0,\id),\; (1,-\id),\; (2,\id),\; (3,\id)}; & S_1 =& \set{(0,\id),\; (1,-\id),\; (2,\id),\; (3,-\id)}; \\ 
	S_5 =& \set{(0,\id),\; (1,\id),\; (2,-\id),\; (3,\id)}; & S_3 =& \set{(0,\id),\; (1,-\id),\; (2,-\id),\; (3,\id)}; \\
	S_6 =& \set{(0,\id),\; (1,\id),\; (2,\id),\; (3,-\id)}; &S_7 =& \set{(0,\id),\; (1,-\id),\; (2,-\id),\; (3,-\id)}.
	\end{align*}
 The corresponding quiver is a union of four connected components $\Kk_0:= \set{S_0}$, $\Kk_1:=\set{S_1}$, $\Kk_{2,3}:=\set{S_2,S_3}$, $\Kk_{4,5,6,7}:=\set{S_4, S_5, S_6, S_7}$.

 The component $\Kk_{4,5,6,7}$ is the following quiver:
 $$
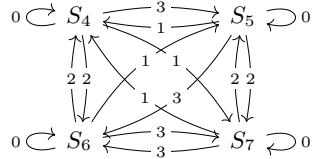
\begin{tikzcd}
     S_4\ar[loop left, "0"]\ar[rr, bend left=10,"3"description] \ar[rrdd, bend left=20,"1"description]\ar[dd, bend left=10,"2"description]&& S_5\ar[loop right, "0"]\ar[ll, bend left=10,"1"description]\ar[dd, bend left=10,"2"description]\ar[lldd, bend left=20,"3"description]\\ && \\
     S_6\ar[loop left, "0"]\ar[rr, bend left=10,"3"description]\ar[uu, bend left=10,"2"description]\ar[rruu, bend left=20,"1"description] && S_7\ar[loop right, "0"]\ar[uu, bend left=10,"2"description]\ar[ll, bend left=10,"3"description]\ar[lluu, bend left=20,"1"description]
 \end{tikzcd}$$
 
 The braiding, restricted on $\Kk_{4,5,6,7}$, acts as follows:
 \begin{align*}
     &[S_4\Vert 3|1]\mmaps [S_4\Vert1|1]     & &[S_4\Vert1|2]\mmaps [S_4\Vert2|1] & &[S_4\Vert2|3]\mmaps[S_4\Vert3|2]\\
     &[S_5\Vert1|3]\mmaps [S_5\Vert3|1] &&[S_5\Vert3|2]\mmaps [S_5\Vert2|1]  &&[S_5\Vert2|3]\mmaps[S_5\Vert1|2]\\
     &[S_6\Vert3|3]\mmaps[S_6\Vert1|3] && [S_6\Vert3|2]\mmaps[S_6\Vert2|3]&&[S_6\Vert1|2]\mmaps[S_6\Vert2|1]\\
     &[S_7\Vert3|3]\mmaps[S_7\Vert1|1] &&[S_7\Vert2|1]\mmaps[S_7\Vert3|2] &&[S_7\Vert2|3]\mmaps[S_7\Vert1|2]\\
     & &&[S_i\Vert a|0]\mmaps[S_i\Vert0|a]&&
 \end{align*}
 and leaves all the other paths unchanged. The family of left quasigroup structures $\set{\bullet_\lambda}_{\lambda}$, for the vertices $\lambda$ in $\Kk_{4,5,6,7}$, is described in Table \ref{table:bullet_K4567}.
 \begin{table}[t]
     \centering
     \begin{tabular}{c|cccc}
$\bullet_{S_4}$ & 0& 1& 2& 3\\ \hline
0 & 0& 1& 2& 3\\
1& 1& 0& 3& 2\\
2& 2& 1& 0& 3\\
3& 3& 0& 1& 2
\end{tabular}\qquad \begin{tabular}{c|cccc}
$\bullet_{S_5}$ & 0& 1& 2& 3\\ \hline
0 & 0& 1& 2& 3\\
1& 1& 2& 3& 0\\
2& 2& 1& 0& 3\\
3& 3& 0& 1& 2
\end{tabular}
\smallskip

\begin{tabular}{c|cccc}
$\bullet_{S_6}$ & 0& 1& 2& 3\\ \hline
0 & 0& 1& 2& 3\\
1& 1& 2& 3& 0\\
2& 2& 3& 0& 1\\
3& 3& 2& 1& 0
\end{tabular}\qquad \begin{tabular}{c|cccc}
$\bullet_{S_7}$ & 0& 1& 2& 3\\ \hline
0 & 0& 1& 2& 3\\
1& 1& 0& 3& 2\\
2& 2& 1& 0& 3\\
3& 3& 2& 1& 0
\end{tabular}
     \caption{The family of left quasigroup structures of $\Kk_{4,5,6,7}$.}
     \label{table:bullet_K4567}
 \end{table}
 The component $\Kk_{2,3}$ is the following quiver:
 $$
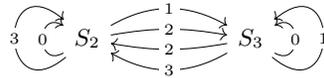
\begin{tikzcd}
     S_2\ar[loop, out=-150, in=150, looseness=3,"0"description]\ar[loop, out=-140, in=140, looseness=6, "3"description]\ar[rr, bend left=10, "2"description]\ar[rr, bend left=30, "1"description] && S_3\ar[loop, out=-30, in=30, looseness=3, "0"description]\ar[loop, out=-40, in=40, looseness=6, "1"description]\ar[ll, bend left=10, "2"description]\ar[ll, bend left=30, "3"description]
 \end{tikzcd}$$
 
 The braiding restricted to $\Kk_{2,3}$ acts as follows:
 \begin{align*}
 &[S_2\Vert 3|3]\mmaps [S_2\Vert 1|3] &&[S_2\Vert 2|3]\mmaps[S_2\Vert 1|2] &&[S_2\Vert 3|2]\mmaps[S_2\Vert 2|1]\\
 &[S_3\Vert 1|1]\mmaps[S_3\Vert 3|1]&&[S_3\Vert 3|2]\mmaps[S_3\Vert 2|1]&&[S_3\Vert 1|2]\mmaps[S_3\Vert 2|3]\\
 & &&[S_i\Vert a|0]\mmaps[S_i\Vert 0|a]&&
 \end{align*}
 and all the other paths are unchanged. The left quasigroup structures $\set{\bullet_\lambda}_\lambda$, for $\lambda$ in $\Kk_{2,3}$, are described in Table \ref{table:bullet_K23}.
 \begin{table}[t]
\begin{center}
\begin{tabular}{c|cccc}
$\bullet_{S_2}$ & 0& 1& 2& 3\\ \hline
0 & 0& 1& 2& 3\\
1& 1& 2& 3& 0\\
2& 2& 1& 0& 3\\
3& 3& 2& 1& 0
\end{tabular}\qquad 
\begin{tabular}{c|cccc}
$\bullet_{S_3}$ & 0& 1& 2& 3\\ \hline
0 & 0& 1& 2& 3\\
1& 1& 0& 3& 2\\
2& 2& 1& 0& 3\\
3& 3& 0& 1&2
\end{tabular}
\end{center}
\caption{The family of left quasigroup structures of $\Kk_{2,3}$}\label{table:bullet_K23}\end{table}
 The braiding restricted to $\Kk_1$ sends
 \begin{align*}
     & [S_1\Vert 1|1]\mmaps [S_1\Vert 3|3]&&[S_1\Vert 1|2]\mmaps[S_1\Vert 2|1]\\
     &[S_1\Vert 2|3]\mmaps[S_1\Vert 3|2]&& [S_1\Vert a|0]\mmaps [S_1\Vert 0|a] 
 \end{align*}and leaves the other paths unchanged. The corresponding left quasigroup operation is described in Table \ref{table:bullet_K1}. 
 \begin{table}[t]
     \centering
     \begin{tabular}{c|cccc}
$\bullet_{S_1}$ & 0& 1& 2& 3\\ \hline
0 & 0& 1& 2& 3\\
1& 1& 0& 3& 2\\
2& 2& 3 &0& 1\\
3& 3& 2& 1&0
\end{tabular}
     \caption{The left quasigroup operation $\bullet_{S_1}$ on $\Kk_1$.}
     \label{table:bullet_K1}
 \end{table}
 Finally, the braiding restricted to $\Kk_0$ is just the canonical flip, with $a\bullet_{S_0} b = a+b$.
\end{example} 
\subsection{On the combinatorics of dynamical skew braces} The following remark provides a hint on the shape of the quiver $\Ss^0_A$. Henceforth, when we say ``the \emph{shape} of a quiver'', we mean the geometry of the directed graph, where the arcs are unlabelled, and the vertices have no name---in other words, the \emph{weak isomorphism class} of the quiver.
\begin{remark}\label{maximal_symmetric_dbrace}
Let $(A,\cdot)$ be a group, and let $\Ss^0_A:= \set{S\subset \Hol(A)\text{ regular }, (1_A,\id_A)\in S}$, associated with the dynamical skew brace $(A, \cdot, \Ss^0_A)$. By Remark \ref{zero_symmetric_string_of_integers}, the shape of the quiver of $(A, \cdot, \Ss^0_A)$ is determined by the string of integers $\set{N^A_s}_s$. Therefore, every group $(A,\cdot)$ is uniquely associated with a string of integers $N^A_s$. These invariants satisfy \begin{equation}\label{eq:relation_NAs}\sum_s sN^A_s = |\Ss^0_A| = |\Aut(A)|^{|A|-1},\end{equation}
and $N^A_s = 0$ if $s\nmid |A|$.
\end{remark}
It is natural to try computing these invariants on small groups. However, the problem of computing $\set{N^A_s}_s$ is very broad, as the following remark shows.
\begin{remark}\label{rem:N_1_is_number_of_sb}
A skew brace is a special kind of dynamical skew brace, with a singleton as the set of vertices. We know from \cite[Theorem 4.2]{guarnieri2017skew} that skew braces $(A,\cdot,\bullet)$ on $(A,\cdot)$ are in a bijective correspondence with regular subgroups of $\Hol(A)$, i.e.\@ with regular dynamical subgroups with a single vertex: these are exactly the isolated vertices of $\Ss^0_A$. In other words, $N^A_{1}$ is the number of operations $\bullet$ such that $(A,\cdot, \bullet)$ is a skew brace. Recall that these skew braces are isomorphic if and only if the corresponding regular subgroups of $\Hol(A)$ are conjugated \cite[Proposition 4.3]{guarnieri2017skew}.
\end{remark}
\begin{example}
    Let $p$ be prime, $A := \Z/p\Z$. Then, one can prove that $N^A_1 = 1$, $N^A_p = \left((p-1)^{(p-1)}-1\right)/p$. This can be proven directly, by showing that $S = \set{(a, \id_A)\mid a\in A}$ is the only isolated vertex. However, Remark \ref{rem:N_1_is_number_of_sb} offers an alternative solution: since $N^A_1$ is the number of brace structures $(\Z/p\Z, +, \bullet)$ on $(\Z/p\Z, +)$, it suffices to observe that $\bullet = +$ is the only possibility. Indeed, $a\mapsto (b\mapsto -a+a\bullet b)$ is a homomorphism $\Z/p\Z\to \Aut(\Z/p\Z)\cong \Z/(p-1)\Z$, hence it must be the constant homomorphism $a\mapsto \id_A$. This is well known, and appears in many classification works on braces and skew braces.
\end{example}
\begin{example}
    Let $A:= \Z/4\Z$. One has $N^A_1 = 2$, $N^A_2= 1$, $N^A_4 = 1$; see Example \ref{ex:Z4}.
\end{example}
\begin{example}
    Let $A:= \Z/2\Z \times \Z/2\Z$. Then, one has $N^A_1 = 4$, $N^A_2 = 6$, $N^A_4 = 50$. The invariants have been computed with \textsc{gap} \cite{GAP4}, and the computation took 3241 milliseconds. Computations get easily prohibitive, even with a computer algebra system, for larger $|A|$ (and especially for larger $|\Aut(A)|$).
\end{example}
We shall henceforth identify dynamical skew braces with their associated braided groupoid. Therefore, we shall speak of ``connected components'' or ``vertices'' of a dynamical skew brace; and this language is made valid through the Matsumoto--Shimizu functor $\Quiv$.
\begin{theorem}\label{thm:maximal_zerosymm_dbrace}
Let $(A,\cdot)$ be a fixed group. Let $(A, \cdot, \Ss)$ be a dynamical skew brace, with $\Ss = \Ss_{\mathrm{in}}\sqcup \Ss_{\mathrm{un}}$ as in Theorem \ref{groupoid_structure_on_Q(A)}, then
    \begin{enumerate}
    \item if $S\in \Ss_{\mathrm{un}}$, and $S'\in\Ss^0_A$ is in the same connected component of $\Ss^0_A$ as $S$, then $S'\in \Ss_{\mathrm{un}}$;
    \end{enumerate}
   i.e., the quiver $Q =\Quiv(A,\phi)$ is obtained as a union of connected components of $(A,\cdot, \Ss^0_A)$, united with the additional set of vertices $\Ss_{\mathrm{in}}$ and the additional set of arrows $Q(\Ss_{\mathrm{in}}, \Ss_{\mathrm{un}})$. Moreover, these arrows from $\Ss_{\mathrm{in}}$ to $\Ss_{\mathrm{un}}$ satisfy the following properties:
    \begin{enumerate}\setcounter{enumi}{1}
        \item for all $S\in \Ss_{\mathrm{in}}$, and for all $S', S''\in \Ss_{\mathrm{un}}$ in the same connected component, there is a bijection $Q(S,S')\cong Q(S', S'')$;
        \item for all $S\in \Ss_{\mathrm{in}}$, and for all $S', S''\in \Ss_{\mathrm{un}}$ in the same connected component, there is a bijection $Q(S, S')\cong Q(S, S'')$.
    \end{enumerate}
\end{theorem}
\begin{proof}
If $S$ is a unital vertex in $\Ss$, then every element $S'\in \Ss^0_A$ in the same connected component as $S$ is obtained as $S'= (a, f)^{-1}S$ for some $(a,f)\in S$: therefore, $S'$ must also belong to $\Ss$. This proves (\textit{i}). The rest follows from Lemma \ref{lemma:semiloopoid_bijections}.
\end{proof}
\begin{figure}[p]
    \centering
    \begin{subfigure}{0.8\linewidth}
        \centering
        \begin{tikzpicture}[x=0.8cm, y=0.8cm]
            \node() at (0,0) {$\bullet$};
            \node() at (2,0) {$\bullet$};
            \node() at (1,1.3) {$\bullet$};
            \node() at (4,1) {$\bullet$};
            \draw[-Stealth] (0,0) to[bend left] (2,0);
            \draw[-Stealth] (2,0) to[bend left] (0,0);
            \draw[-Stealth] (0,0) to[bend left] (1,1.3);
            \draw[-Stealth] (1,1.3) to[bend left] (0,0);
            \draw[-Stealth] (1,1.3) to[bend left] (2,0);
            \draw[-Stealth] (2,0) to[bend left] (1,1.3);
            \draw[-Stealth] (0,0) arc[radius=.3, start angle=45, end angle=405];
            \draw[-Stealth] (2,0) arc[radius=.3, start angle=135, end angle=500];
            \draw[-Stealth] (1,1.3) arc[radius=.3, start angle=-90, end angle=270];
            \draw[-Stealth] (4,1) arc[radius=.3, start angle=45, end angle=405];
            \draw[-Stealth] (4,1) arc[radius=.5, start angle=45, end angle=405];
            \draw[-Stealth] (4,1) arc[radius=.7, start angle=45, end angle=405];
            \node(mark) at (5.5,0) {{\huge\ding{51}}};
            \node[white]() at (-2,0) {};
        \end{tikzpicture}
        \caption{The quiver associated with the maximal zero-symmetric dynamical skew brace over $A = \Z/3\Z$; see Example \ref{ex:Z3}.}\label{subfig:maximal_zerosymm}
    \end{subfigure}
    
    \begin{subfigure}{0.8\linewidth}
        \centering
        \begin{tikzpicture}[x=0.8cm, y=0.8cm]
        \node() at (0,0) {$\bullet$};
            \node() at (2,0) {$\bullet$};
            \node() at (1,1.3) {$\bullet$};
            \draw[-Stealth] (0,0) to[bend left] (2,0);
            \draw[-Stealth] (2,0) to[bend left] (0,0);
            \draw[-Stealth] (0,0) to[bend left] (1,1.3);
            \draw[-Stealth] (1,1.3) to[bend left] (0,0);
            \draw[-Stealth] (1,1.3) to[bend left] (2,0);
            \draw[-Stealth] (2,0) to[bend left] (1,1.3);
            \draw[-Stealth] (0,0) arc[radius=.3, start angle=45, end angle=405];
            \draw[-Stealth] (2,0) arc[radius=.3, start angle=135, end angle=500];
            \draw[-Stealth] (1,1.3) arc[radius=.3, start angle=-90, end angle=270];
            \node(mark) at (5.5,0) {{\huge\ding{51}}};
            \node[white]() at (-2,0) {};
        \end{tikzpicture}
        \caption{The quiver of an admissible zero-symmetric dynamical skew brace over $A = \Z/3\Z$.}\label{subfig:admissible_zerosymm}
    \end{subfigure}

    \begin{subfigure}{0.8\linewidth}
        \centering
        \begin{tikzpicture}[x=0.8cm, y=0.8cm]
            \node() at (0,0) {$\bullet$};
            \node() at (2,0) {$\bullet$};
            \node() at (1,1.3) {$\bullet$};
            \node[gray]() at (4,1.5) {$\bullet$};
            \node[gray]() at (4.3,1) {$\bullet$};
            \draw[-Stealth] (0,0) to[bend left] (2,0);
            \draw[-Stealth] (2,0) to[bend left] (0,0);
            \draw[-Stealth] (0,0) to[bend left] (1,1.3);
            \draw[-Stealth] (1,1.3) to[bend left] (0,0);
            \draw[-Stealth] (1,1.3) to[bend left] (2,0);
            \draw[-Stealth] (2,0) to[bend left] (1,1.3);
            \draw[-Stealth] (0,0) arc[radius=.3, start angle=45, end angle=405];
            \draw[-Stealth] (2,0) arc[radius=.3, start angle=135, end angle=500];
            \draw[-Stealth] (1,1.3) arc[radius=.3, start angle=-90, end angle=270];
            \draw[-Stealth, color=gray] (4,1.5) to (0,0);
            \draw[-Stealth, color=gray] (4,1.5) to (2,0);
            \draw[-Stealth, color=gray] (4,1.5) to (1,1.3);
            \draw[-Stealth, color=gray] (4.3,1) to (0,0);
            \draw[-Stealth, color=gray] (4.3,1) to (2,0);
            \draw[-Stealth, color=gray] (4.3,1) to (1,1.3);
            \node(mark) at (5.5,0) {{\huge\ding{51}}};
            \node[white]() at (-2,0) {};
        \end{tikzpicture}
        \caption{An admissible quiver for some dynamical skew brace over $A = \Z/3\Z$: namely, the one with unital vertices $\set{(0,\id), (1,\id), (2, -\id)}$, $\set{(0,\id), (1,-\id), (2, \id)}$, $\set{(0,\id), (1,-\id), (2, -\id)}$, and initial vertices $\set{(0,-\id), (1,\id), (2, \id)}$, $\set{(0,-\id), (1, -\id), (2,\id)}$.}\label{subfig:admissible_nonzerosymm}
    \end{subfigure}

\begin{subfigure}{0.8\linewidth}
        \centering
        \begin{tikzpicture}[x=0.8cm, y=0.8cm]
            \node() at (4,1) {$\bullet$};
            \draw[-Stealth] (4,1) arc[radius=.3, start angle=45, end angle=405];
            \draw[-Stealth] (4,1) arc[radius=.5, start angle=45, end angle=405];
            \draw[-Stealth] (4,1) arc[radius=.7, start angle=45, end angle=405];
            \node[gray]() at (1,0.5) {$\bullet$};
            \draw[-Stealth, color=gray] (1,0.5) to[bend left] (4,1);
            \draw[-Stealth, color=gray] (1,0.5) to[bend right] (4,1);
            \draw[-Stealth, color=gray] (1,0.5) to (4,1);
            \node(mark) at (5.5,0) {{\huge\ding{51}}};
            \node[white]() at (-2,0) {};
        \end{tikzpicture}
        \caption{An admissible quiver for some dynamical skew brace over $A = \Z/3\Z$: namely, the one with unital vertex $\set{(0,\id), (1,\id), (2, \id)}$, and initial vertex $\set{(0,-\id), (1,-\id), (2, -\id)}$.}\label{subfig:other_admissible_nonzerosymm}
    \end{subfigure}

\begin{subfigure}{0.8\linewidth}
        \centering
        \begin{tikzpicture}[x=0.8cm, y=0.8cm]
            \node() at (4,1) {$\bullet$};
            \draw[-Stealth] (4,1) arc[radius=.3, start angle=45, end angle=405];
            \draw[-Stealth] (4,1) arc[radius=.5, start angle=45, end angle=405];
            \draw[-Stealth] (4,1) arc[radius=.7, start angle=45, end angle=405];
            \node[gray]() at (1,0.5) {$\bullet$};         
            \draw[-Stealth, color=gray] (1,0.5) to (4,1);
            \node(mark) at (5.5,0) {{\huge\ding{55}}};
            \node[white]() at (-2,0) {};
        \end{tikzpicture}
        \caption{This quiver is not admissible as the quiver of a dynamical skew brace: it violates Theorem \ref{thm:maximal_zerosymm_dbrace} (\textit{ii}).}\label{subfig:nonadmissible_iii}
    \end{subfigure}

    \begin{subfigure}{0.8\linewidth}
        \centering
        \begin{tikzpicture}[x=0.8cm, y=0.8cm]
            \node() at (0,0) {$\bullet$};
            \node() at (2,0) {$\bullet$};
            \node() at (1,1.3) {$\bullet$};
            \node[gray]() at (4,1.5) {$\bullet$};
            \draw[-Stealth] (0,0) to[bend left] (2,0);
            \draw[-Stealth] (2,0) to[bend left] (0,0);
            \draw[-Stealth] (0,0) to[bend left] (1,1.3);
            \draw[-Stealth] (1,1.3) to[bend left] (0,0);
            \draw[-Stealth] (1,1.3) to[bend left] (2,0);
            \draw[-Stealth] (2,0) to[bend left] (1,1.3);
            \draw[-Stealth] (0,0) arc[radius=.3, start angle=45, end angle=405];
            \draw[-Stealth] (2,0) arc[radius=.3, start angle=135, end angle=500];
            \draw[-Stealth] (1,1.3) arc[radius=.3, start angle=-90, end angle=270];
            \draw[-Stealth, color=gray] (4,1.5) to (0,0);
            \draw[-Stealth, color=gray] (4,1.5) to (1,1.3);
            \node(mark) at (5.5,0) {{\huge\ding{55}}};
            \node[white]() at (-2,0) {};
        \end{tikzpicture}
        \caption{This quiver is not admissible as the quiver of a dynamical skew brace: it violates Theorem \ref{thm:maximal_zerosymm_dbrace} (\textit{ii}) and (\textit{iii}).}\label{subfig:nonadmissible_iv}
    \end{subfigure}
    \caption{Some examples of admissible and nonadmissible quivers for dynamical skew braces over $A = \Z/3\Z$.}
    \label{fig:admissible_quivers}
\end{figure}
A number of examples is depicted in Figure \ref{fig:admissible_quivers}. This makes clear that a quiver, if it is realised by some dynamical skew brace over $(A,\cdot)$, must be obtained by adding initial vertices and arrows to some subquiver of $\Ss^0_A$. Conditions (\textit{ii}) and (\textit{iii}) from Theorem \ref{thm:maximal_zerosymm_dbrace} are necessary for a quiver to be a subquiver of $\Ss_A$ associated with some dynamical skew brace, but not sufficient; see Figure \ref{fig:iii-iv-not-sufficient}.
\begin{figure}[t]
    \centering
  
\begin{tikzpicture}[x=0.8cm, y=0.8cm]
            \node() at (0,0) {$\bullet$};
            \node() at (2,0) {$\bullet$};
            \node() at (1,1.3) {$\bullet$};
            \node() at (7.5,1) {$\bullet$};
            \node[gray]() at (4,1.5) {$\bullet$};
            \node[gray]() at (5,0.5) {$\bullet$};
            \node[gray]() at (4.3,1) {$\bullet$};
            \node[gray]() at (5.2,0.2) {$\bullet$};
            \draw[-Stealth] (0,0) to[bend left] (2,0);
            \draw[-Stealth] (2,0) to[bend left] (0,0);
            \draw[-Stealth] (0,0) to[bend left] (1,1.3);
            \draw[-Stealth] (1,1.3) to[bend left] (0,0);
            \draw[-Stealth] (1,1.3) to[bend left] (2,0);
            \draw[-Stealth] (2,0) to[bend left] (1,1.3);
            \draw[-Stealth] (0,0) arc[radius=.3, start angle=45, end angle=405];
            \draw[-Stealth] (2,0) arc[radius=.3, start angle=135, end angle=500];
            \draw[-Stealth] (1,1.3) arc[radius=.3, start angle=-90, end angle=270];
            \draw[-Stealth, color=gray] (4,1.5) to (0,0);
            \draw[-Stealth, color=gray] (4,1.5) to (2,0);
            \draw[-Stealth, color=gray] (4,1.5) to (1,1.3);
            \draw[-Stealth, color=gray] (4.3,1) to (0,0);
            \draw[-Stealth, color=gray] (4.3,1) to (2,0);
            \draw[-Stealth, color=gray] (4.3,1) to (1,1.3);
            \draw[-Stealth] (7.5,1) arc[radius=.3, start angle=45, end angle=405];
            \draw[-Stealth] (7.5,1) arc[radius=.5, start angle=45, end angle=405];
            \draw[-Stealth] (7.5,1) arc[radius=.7, start angle=45, end angle=405];
            \draw[-Stealth, color=gray] (5,0.5) to (7.5,1);
            \draw[-Stealth, color=gray] (5,0.5) to[bend left] (7.5,1);
            \draw[-Stealth, color=gray] (5,0.5) to[bend right] (7.5,1);
            \draw[-Stealth, color=gray] (5.2,0.2) to (7.5,1);
            \draw[-Stealth, color=gray] (5.2,0.2) to[bend left] (7.5,1);
            \draw[-Stealth, color=gray] (5.2,0.2) to[bend right] (7.5,1);
            \node(mark) at (8.5,0) {{\huge\ding{55}}};
            \node[white]() at (-2,0) {};
        \end{tikzpicture}
    \caption{Although this quiver has an admissible number of vertices, and technically satisfies conditions (\textit{ii}) and (\textit{iii}) of Theorem \ref{thm:maximal_zerosymm_dbrace}, it does not correspond to any dynamical skew brace. Indeed, it is not a subquiver of the one in Example \ref{ex:Z3_initials}.}\label{fig:iii-iv-not-sufficient}
\end{figure}
\begin{corollary}
    Let $A$ be a finite group. Then, the shape of the quiver $\Ss_A$ is uniquely determined by two data: the quiver $\Ss^0_A$, and, for each connected component $\Kk$ of $\Ss^0_A$, the number $\mathrm{in}^A_\Kk$ of initial vertices $R$ with arrows into $\Kk$.
\end{corollary}
\begin{proof}
    If $R$ is an initial vertex in $\Ss^0_A$, then it has at least one arrow into some connected component $\Kk$ of $\Ss^0_A$. By Theorem \ref{thm:maximal_zerosymm_dbrace}, then, all arrows starting from $R$ must go into $\Kk$, and for each $S\in \Obj(\Kk)$ there are exactly $|A|/|\Obj(\Kk)|$ arrows from $R$ to $S$. Thus, in order to draw the shape of the quiver $\Ss_A$, it suffices to know the numbers $\mathrm{in}^A_\Kk$ for each connected component $\Kk$ of $\Ss^0_A$.
\end{proof}
Unlike the $N^A_s$'s, the numbers $\mathrm{in}^A_\Kk$ can easily be computed in explicit form.
\begin{proposition}\label{prop:in^A_K}
    For a connected component $\Kk$ of $\Ss^0_A$, with $s$ vertices, one has
    \[ \mathrm{in}^A_\Kk= s (|\Aut(A)|-1). \]
\end{proposition}
\begin{proof}
    Let $S$ be a vertex in $\Kk$. Let $\tilde{\Kk}$ be the connected component of $S$ in $\Ss_A$, which includes $\Kk$ as a subquiver. All distinct incoming arrows in $S$ are obtained in one of the following ways:
    \begin{enumerate}
        \item An incoming arrow $R\to S$ with $R$ unital is obtained by the relation $S = (a,f)^{-1}R $ for some $(a,f)\in R$. The element $(a,f)^{-1} = (a,f)^{-1}(1_A,\id_A)$ lies in $S$. One has $R = (b,g)S$.
        \item An incoming arrow $R\to S$ with $R$ initial is obtained by the relation $S = (a,f)^{-1}R$ for $(a,f)\in R$, but now $(a,f)^{-1}$ does not lie in $S$. Observe that the vertex $R$ does not contain $(1_A, \id_A)$, but it does contain $(1_A, h)$ for some automorphism $h\in \Aut(A)\smallsetminus\set{\id_A}$. Thus there is an element $(b,g):= (a,f)^{-1}(1_A, h)$ in $S$ such that $R$ is retrieved as $R = (b,g)S$.
    \end{enumerate}

    There are $|A|$ incoming arrows of type (\textit{i}): as many as the elements $(b,g)$ in $ S$. On the other hand, there are $|A| (|\Aut(A)|-1)$ incoming arrows of type (\textit{ii}): for each $(b,g)\in S$, there are as many arrows as the automorphisms $h\neq \id_A$. 

    Now let $R$ be an initial vertex in $\tilde{\Kk}$, and $S$ a unital vertex in $\Kk$. By Theorem \ref{thm:maximal_zerosymm_dbrace}, the number of arrows $R\to S$ is the same as the degree of the complete quiver $\Kk$, that is $|A|/s$. Therefore, the number of unital vertices in $\tilde{\Kk}$ is 
    \[ \mathrm{in}^A_\Kk = |A|(|\Aut(A)|-1) \cdot \frac{s}{|A|}= s(|\Aut(A)|-1). \]
\end{proof}
In particular, observe that the number $\mathrm{in}^A_\Kk$ does not depend on the braided groupoid structure of $\Kk$, but only on its number of vertices (or, equivalently, on its degree as a complete quiver).
\begin{example}
    In the case $A = \Z/3\Z$, one has $|\Aut(A)|= 2$. The two connected components $\Kk_0$ and $\Kk_1$ from Example \ref{ex:Z3} have indeed $\mathrm{in}^A_{\Kk_0} = 1$, $\mathrm{in}^A_{\Kk_1}= 3$; see Example \ref{ex:Z3_initials}.
\end{example}
\begin{convention}
    Since the numbers $\mathrm{in}^A_\Kk$ do not depend on $\Kk$, but only on its number of vertices $s$, we use the notation $\mathrm{in}^A_s:= \mathrm{in}^A_{\Kk}$.
\end{convention}
\begin{remark}
    As a double-check, let us retrieve the relation $|\Ss_A| = |\Aut(A)|^{|A|}$ from \eqref{eq:relation_NAs} and Proposition \ref{prop:in^A_K}. The number of objects of $\Ss_A$ is clearly given by $\sum_{s\mid \, |A|} (sN^A_s+\mathrm{in}^A_sN^A_s)$, where the first term is the number of unital vertices in all the components with $s$ vertices, and the second term accounts for the number of initial vertices. Now, using Proposition \ref{prop:in^A_K}, this is rewritten as $\sum_{s\mid\, |A|}N^A_s(s + s(|\Aut(A)|-1)) = \sum_{s\mid\, |A|}sN^A_s(|\Aut(A)|)$, which by \eqref{eq:relation_NAs} equals \[|\Aut(A)|^{|A|-1}\cdot |\Aut(A)| = |\Aut(A)|^{|A|},\] as desired. 
\end{remark}
Let $A$ be a finite group, and $S$ be a vertex in $\Ss_A$. For each automorphism $f\in \Aut(A)$, we can count how many pairs of the form $(a,f)$ lie in $S$: we denote this number by $n^S_f$. The set $\set{n^S_f\mid f\in \Aut(A)}$ can be arranged in decreasing order, forming a numerical partition of $|A|$ that we call $\mathrm{part}(S)$. 

For instance, let $A =\Z/4\Z$ and consider $S_7 = \set{(0,\id), (1, -\id), (2, -\id), (3, -\id)}$ in $\Ss_A$ as in Example \ref{ex:Z4}. There is one occurrence of $\id$, and there are three occurrences of $-\id$, thus $\mathrm{part}(S_7) = (3,1)$ which is a partition of $4 = |A|$.
\begin{proposition}\label{prop:permutations}
    The partition $\mathrm{part}(S)$ depends exclusively on the connected component of $\Ss_A$ that contains $S$.
\end{proposition}
\begin{proof}
    Let $\tilde{\Kk}$ be the connected component of $\Ss_A$ containing $S$. When we obtain a vertex $S'$ in $\tilde{\Kk}$ as $S' = (a,f)^{-1}S$ for some $(a,f)\in S$, the second component of each pair $(a,f)^{-1}(b,g) \in  S'$ is given by $f^{-1}g$. Notice that $f^{-1}g = f^{-1}g'$ if and only if $g = g'$, whence $\mathrm{part}(S) = \mathrm{part}(S')$. This proves the claim for all unital vertices.

    If now $R$ is an initial vertex in $\tilde{\Kk}$, there is a unital vertex $S'$ in $ \tilde{\Kk}$ such that $S' = (a,f)^{-1}R$ for some $(a,f)\in R$. Thus $R = (a,f)S'$, and the above argument can be repeated \emph{verbatim}.
\end{proof}
\begin{remark}
    If $S$ and $S'$ lie in the same connected component, then $\mathrm{part}(S) = \mathrm{part}(S')$ by Proposition \ref{prop:permutations}, but \emph{the converse is false}. Indeed, in Example \ref{ex:Z4}, $\mathrm{part}(S_1) = \mathrm{part}(S_2) = (2,2)$, but $S_1$ and $S_2$ lie in distinct components.
\end{remark}

\section{Skew bracoids as dynamical skew braces}\label{sec:aredbraces}
\noindent We now construct an isomorphism between every skew bracoid, and a suitable zero-symmetric dynamical skew brace, thus proving that all braided groupoids originate from dynamical skew braces up to isomorphism. 

Skew bracoids have a suggestive geometric meaning, as discrete ``tangent bundles'' where the groupoid $\Gg$ plays the role of a manifold, and the group structure $(\Gg(\lambda,\Lambda)\cdot_\lambda)$ is the ``tangent space'' at $\lambda$; this is more deeply motivated in \cite[\S6]{sheng2024postgroupoidsquivertheoreticalsolutionsyangbaxter}. In this framework, our Matsumoto--Shimizu labellings can be seen as a ``parallelisation'' of $\Gg$ in a way that is consistent with the groupoid structure.

\subsection{On the Matsumoto--Shimizu functor}We begin with an obvious remark. The Matsumoto--Shimizu functor $\Quiv$ is fully faithful strong monoidal, and hence\footnote{This implication requires the Axiom of Choice.} its corestriction to the essential image $\im(\Quiv)$ is a monoidal equivalence between $\im(\Quiv)$ and $\mathsf{Set}_\Lambda$. In particular, there is an inverse functor $\mathcal{R}\colon \im(\Quiv)\to \mathsf{Set}_\Lambda$ such that $\mathcal{Q}\circ \mathcal{R} \cong \id_{\mathsf{Set}_\Lambda}$ and $\mathcal{R}\circ \mathcal{Q}\cong \id_{\im(\Quiv)}$. 

\begin{definition}
	Let $Q$ be a quiver over $\Lambda$, and suppose that $A$ is a set in bijection with $Q(\lambda,\Lambda)$ for all $\lambda\in\Lambda$. A \emph{labelling} on $Q$ with values in $A$ is a family of bijections $\set{\varphi_\lambda\colon Q(\lambda,\Lambda)\to A}_{\lambda\in\Lambda}$.
\end{definition}

If $Q$ is a quiver in $\im(\Quiv)$, i.e., if there is a set $A$ such that $Q(\lambda,\Lambda)$ is in bijection with $A$  for all $\lambda\in\Lambda$, then defining a labelling on $Q$ with values in $A$ is the same as finding a dynamical set $(A,\phi)$ such that $\Quiv(A,\phi)\cong Q$.

If we want the functor $\mathcal{R}$ to be monoidal, we can choose the labellings in a way that is compatible with tensor products. If $Q$ is labelled by a dynamical set $(A,\phi)$, and $R$ is labelled by $(B,\psi)$, then $Q\ot R$ is naturally labelled by $(A,\phi)\ot (B,\psi)$, by identifying $[\lambda\Vert a]\ot [\phi(\lambda, a)\Vert b]$ with $[\lambda\Vert (a,b)]$.

Many quiver-theoretic concepts in $\im(\Quiv)$ can be read in $\mathsf{Set}_\Lambda$ through labellings. The following is a list of immediate examples.
\subsubsection{} The oidification of semigroups is called \textit{semicategory} \cite{DominionOfIsbell}. A small semicategory is a quiver $Q$ with an associative binary operation $\bullet\colon Q\ot Q\to Q$, and it is clear that in general this need not be contained in $\im(\Quiv)$. However, if $\bullet$ is a semicategory structure on $Q = \Quiv(A,\phi)$ for some dynamical set $(A,\phi)$ over $\Lambda$, then we can read it as a family $\set{\bullet_\lambda}_{\lambda\in\Lambda}$ of binary operations on $A$, through the assignment  $[\lambda\Vert a]\bullet [\phi(\lambda, a)\Vert b] = [\lambda\Vert a\bullet_\lambda b]$. The associativity of $\bullet$ translates as the dynamical associativity \eqref{eq:dyn-associativity}.
\subsubsection{} The oidification of a monoid is a (small) \textit{category}. It is easy to construct examples of categories that do not lie in $\im(\Quiv)$. However, if $\Cc = \Quiv(A,\phi)$ is a category over $\Obj(\Cc) = \Lambda$ that lies in $\im(\Quiv)$, and has binary composition $\bullet\colon \Cc\ot \Cc\to \Cc$ and family of units $\set{\mathbf{1}_\lambda}_{\lambda \in \Lambda}$, then $\bullet $ can be read on $A$ as a family of binary operations $\set{\bullet_{\lambda}}_{\lambda\in\Lambda}$ satisfying \eqref{eq:dyn-associativity}, and $\set{\mathbf{1}_\lambda}_{\lambda\in\Lambda}$ induces a family $\set{1_\lambda}_{\lambda\in\Lambda}\subseteq A$ through the assignment $\mathbf{1}_\lambda = [\lambda\Vert 1_\lambda]$, with the properties $a\bullet_{\lambda} 1_{\phi(\lambda, a)} = a$ and $1_\lambda\bullet_\lambda a = a$ for all $a\in A$, $\lambda\in\Lambda$. 
\subsubsection{} Connected left unital associative semiloopoids always lie in $\im(\Quiv)$: this follows from the left-multiplication maps being bijective. Let $\bullet$ be the semiloopoid operation on $\Gg = \Quiv(A,\phi)$ and let $\set{\bullet_{\lambda} }_{\lambda\in\Lambda}$ be retrieved as above. This family of operations satisfies \eqref{eq:dyn-associativity}. Moreover, since every left multiplication in $\Gg$ is bijective, it follows that every $\bullet_\lambda$ is a left quasigroup structure on $A$.
\subsubsection{} Groupoids are the oidification of groups, and connected groupoids always lie in $\im(\Quiv)$. A groupoid is a left unital associative semiloopoid over $\Lambda = \Lambda_{\mathrm{un}}$, thus a groupoid operation $\bullet$ corresponds with a family of left and right quasigroup operations $\bullet_\lambda$ satisfying \eqref{eq:dyn-associativity}. Since a groupoid is also a category, the family of units $\mathbf{1}_\lambda$ correspond to a family of elements $1_\lambda \in A$ satisfying $a\bullet_{\lambda} 1_{\phi(\lambda, a)} = a$ and $1_\lambda\bullet_\lambda a = a$ for all $a\in A$, as above. One clearly has $1_\lambda = a\backslash_\lambda a$ for all $a\in A$.
\subsubsection{} A skew bracoid $\Gg = \Quiv(A,\phi)$ has two interlaced structures: a groupoid structure $(\Gg,\bullet)$ and a family of group structures $(\Gg(\lambda,\Lambda), \cdot_\lambda)$. The groupoid o\-pe\-ra\-tion corresponds to a family $\set{\bullet_\lambda}_{\lambda\in\Lambda}$ as above. The group operations $\cdot_\lambda$ correspond to a family of group operations $\tilde{\cdot}_\lambda$ on $A$. We could write down the axioms for this family of group operations, and define ``by hand'' an algebraic structure in $\mathsf{Set}_\Lambda$ that is equivalent to skew bracoids by construction. 

But we may actually wish for something more. We may hope that all the $\cdot_\lambda$'s induce the same group structure on $A$. Unfortunately, this cannot be guaranteed in general, as the following counterexample shows.
\begin{counterexample}\label{counterex:neednotcoincide}
Consider the skew bracoid $\Kk_{2,3}$ from Example \ref{ex:Z4}. We choose a different labelling, and relabel the edges as follows:
 $$\begin{tikzcd}
	S_2\ar[loop, out=-150, in=150, looseness=3,"0"description]\ar[loop, out=-140, in=140, looseness=6, "3"description]\ar[rr, bend left=10, "2"description]\ar[rr, bend left=30, "1"description] && S_3\ar[loop, out=-30, in=30, looseness=3, "0"description]\ar[loop, out=-40, in=40, looseness=6, "1"description]\ar[ll, bend left=10, "2"description]\ar[ll, bend left=30, "3"description]
\end{tikzcd} \quad \to\quad \begin{tikzcd}
S_2\ar[loop, out=-150, in=150, looseness=3,"0"description]\ar[loop, out=-140, in=140, looseness=6, "\mathbf{2}"description]\ar[rr, bend left=10, "\mathbf{3}"description]\ar[rr, bend left=30, "1"description] && S_3\ar[loop, out=-30, in=30, looseness=3, "0"description]\ar[loop, out=-40, in=40, looseness=6, "1"description]\ar[ll, bend left=10, "2"description]\ar[ll, bend left=30, "3"description]
\end{tikzcd}$$
Now the induced operations $\bullet_{S_3}$ and $\tilde{\cdot}_{S_3}$ are the same as in Example \ref{ex:Z4}, but the operations $\bullet_{S_2}$ and $\tilde{\cdot}_{S_2}$ have changed by applying the permutation $(2\, 3)$ to the tables of the corresponding operations from Example \ref{ex:Z4}. In particular, $\tilde{\cdot}_{S_3}$ is the sum in $\Z/4\Z$, but $\tilde{\cdot}_{S_2}$ is not any more.
\end{counterexample}
As we see from the above counterexample, it is easy to start from a dynamical skew brace with its usual labelling, and ``mess up'' with the labelling so that the induced operations do not coincide any more.

For the rest of the paper, our aim will be to perform the converse: namely, to start from a skew bracoid, and to ``sort'' the labelling so that the induced operations coincide.
\subsection{Matsumoto--Shimizu labellings}

Let $\Gg$ be a skew bracoid on a homogeneous quiver of weight $n$. Every set $\Gg(\lambda,\Lambda)$ has cardinality $n$, and hence can be put in bijection with a chosen set $A$ of cardinality $n$. Let $\varphi$ be a labelling on $\Gg$ with values in $A$. Each operation $\cdot_\lambda$ on $\Gg(\lambda,\Lambda)$ induces a group structure $\tilde{\cdot}_\lambda$ on $A$, as above. We begin with a crucial remark.
\begin{remark}As we have seen in Counterexample \ref{counterex:neednotcoincide}, the operations $\tilde{\cdot}_\lambda$ need not coincide. But \emph{if they do}---i.e., if all the $\cdot_\lambda$'s induce the same group operation $\cdot$ on $A$---then the compatibility \eqref{qtbc} between $\set{\cdot_\lambda}_{\lambda\in\Lambda}$ and $\set{\bullet_\lambda}_{\lambda\in\Lambda}$ implies the compatibility \eqref{eq:brace-compatibility} between $\cdot$ and $\set{\bullet_\lambda}_{\lambda\in\Lambda}$; and hence $\Gg$ is in fact (isomorphic to) a dynamical skew brace.\end{remark}

Thus the following definition is very natural, and the subsequent lemma is straightforward.
\begin{definition}\label{def:ms_labellings}
    Let $\Gg$ be a skew bracoid over $\Lambda$. A \emph{Matsumoto--Shimizu labelling} on $\Gg$ is a labelling $\set{\varphi_\lambda\colon \Gg(\lambda,\Lambda)\to A}_{\lambda\in \Lambda}$, with values in a set $A$, such that the binary operation $a\cdot b := \varphi_\lambda (\varphi_\lambda^{-1}(a)\cdot_\lambda \varphi_\lambda^{-1}(b)) $ on $A$ induced by $\cdot_\lambda$ is independent of $\lambda\in\Lambda$.
\end{definition}
\begin{lemma}\label{lemma:mslabelling-iff-dbrace}
Given a skew bracoid $\Gg$, the following data are equivalent:
\begin{enumerate}
\item a Matsumoto--Shimizu labelling $\varphi$ on $\Gg$;
\item an isomorphism of skew bracoids between $(\Gg,\set{\cdot_\lambda}_{\lambda\in\Lambda},\bullet)$ and the skew bracoid associated with some zero-symmetric dynamical skew brace.
\end{enumerate}
In other words, $\Gg$ has a Matsumoto--Shimizu labelling if and only if $\Gg$ is ``morally'' a dynamical skew brace.
\end{lemma}
\begin{proof}Let $\varphi$ be a Matsumoto--Shimizu labelling on $\Gg$, and let $A$ be the set of labels. Then, the operations $\cdot_\lambda$ all induce the same binary operation $\cdot $ on $A$, which is clearly a group operation. In particular all the $\tilde{\cdot}_\lambda$'s have the same unit, namely the element $1_A:= \varphi_\lambda(\mathbf{1}_\lambda)\in A$, which is thereby independent of $\lambda$. The left quasigroup operations $\bullet_\lambda$ are defined by the request \[a\bullet_\lambda b = \varphi_\lambda\left(\varphi^{-1}_\lambda(a)\bullet \varphi^{-1}_{\target\left(\varphi^{-1}_\lambda(a)\right)}(b)\right),\]
	for all $a,b\in A$, $\lambda\in\Lambda$. With these operations, $\Gg$ clearly becomes a dynamical skew brace. The map $x\mapsto  [\source(x)\Vert  \varphi_{\source(x)}(x)]$ provides the desired isomorphism between the quiver-theoretic skew brace structure, and (the quiver-theoretic skew brace associated with) the dynamical skew brace structure.

Conversely, suppose given an isomorphism $f$ between $\Gg$ and a dynamical skew brace. Define a labelling $\varphi$ by $\varphi_\lambda(f^{-1}([\lambda\Vert a])):= a$. This is clearly a Matsumoto--Shimizu labelling. 
\end{proof}
In order to construct Matsumoto--Shimizu labellings, we shall need the following result.
\begin{remark}\label{rem:maximal_le1_subgroupoid}Every groupoid $\Gg$ over $\Lambda$ has at least one maximal subgroupoid $\Gg'$ that is Schurian and contains $\set{\mathbf{1}_\lambda}$ for all $\lambda$. The family of all subgroupoids $\Hh$ satisfying the above properties is a poset with inclusion, and every chain has a supremum given by the union, thus the existence of such a $\Gg'$ is granted by Zorn's Lemma. Notice that $\Gg'$ is wide, and clearly a normal subgroupoid of $\Gg$ in the sense of Paques and Tamusiunas \cite{paques2018galois}. 
\end{remark}
\begin{lemma}\label{lem:maximal_PH_subgroupoid}
    Let $\Gg$ be a groupoid over $\Lambda$, and $\Gg'$ be as in Remark \ref{rem:maximal_le1_subgroupoid}. If $\Gg$ is connected (and hence complete of degree $d$, by Remark \ref{quiver_of_groupoid_is_homogeneous}), then $\Gg'$ is a groupoid of pairs.
\end{lemma}
\begin{proof} Let $\Gg'$ be as above. We prove that, for all pairs $(\lambda,\mu)\in\Lambda_{\mathrm{un}}\times \Lambda_{\mathrm{un}}$, there exists an arrow in $\Gg'(\lambda,\mu)$. This holds trivially when $|\Lambda| = 1$, thus we assume $|\Lambda|>1$.

    Suppose by contradiction that $\Gg'$ is not a groupoid of pairs. This means that there exist $\lambda\neq \mu\in\Lambda$ such that $\Gg'(\lambda,\mu) =\emptyset$. But $\Gg(\lambda,\mu)$ contains at least one element $x$. Thus we form a new subgroupoid $\Gg''$ containing $\Gg'$, in the following way: we attach to $\Gg'$ the arrows $x$ and $x^{-\bullet}$; and, for all arrows $x_{\mu,\nu}\in \Gg'(\mu,\nu)$, we attach the arrows $x\bullet x_{\mu,\nu}\in \Gg''(\lambda,\nu)$ and $x_{\mu,\nu}^{-\bullet}\bullet x^{-\bullet}\in \Gg''(\nu,\lambda)$. The result is again a subgroupoid $\Gg''$ of $\Gg$, which strictly contains $\Gg'$, and is Schurian. This contradicts the maximality of $\Gg'$.
\end{proof}

\subsection{Constructing Matsumoto--Shimizu labellings}\label{sec:constructing} For the rest of this section, let $\Gg$ be a braided groupoid $(\Gg,\bullet)$ with braiding $\sigma$; or, equivalently, a skew bracoid structure on $\Gg$. Moreover, assume that $\Gg$ is connected, and hence complete, of some degree $d$. We now construct a Matsumoto--Shimizu labelling $\varphi$, thus proving that $\Gg$ is isomorphic to a (zero-symmetric) dynamical skew brace.

Since the braiding $\sigma$ is non-degenerate, for all pairs of vertices $\lambda,\mu$ connected by an arrow $x\colon \lambda\to \mu$, there is a bijection $x^{-\bullet}\rightharpoonup\blank\colon \Gg(\lambda,\Lambda)\to \Gg(\mu,\Lambda)$. We proceed as follows:
\begin{enumerate}
\item We fix an initial vertex $\zeta$, and a bijection $\varphi_\zeta\colon \Gg(\zeta,\Lambda)\to A$.
\item For all $\lambda$ such that there is an arrow $x\colon \zeta\to \lambda$, we choose as $\varphi_\lambda$ the unique bijection $\varphi_\lambda\colon \Gg(\lambda,\Lambda)\to A$ that makes the diagram below commute:
$$\begin{tikzcd}
\Gg(\zeta,\Lambda)\ar[rd,"\varphi_\zeta" {below left}]\ar[rr,"x^{-\bullet}\rightharpoonup\blank"]&&\Gg(\lambda,\Lambda)\ar[ld,"\varphi_\lambda"{below right}]\\
& A &
\end{tikzcd}$$
\end{enumerate} Since $\Gg$ is connected, the operation (\textit{ii}) can be carried out for all $\lambda$. The arrows $x\colon \zeta\to \lambda$ that we use in (\textit{ii}) are not unique in general, and if $x$ and $y$ are two arrows $\zeta\to \lambda$ the maps $x^{-\bullet}\rightharpoonup\blank$ and $y^{-\bullet}\rightharpoonup\blank$ are generally distinct\footnote{For instance if $\Gg$ is a braided \emph{group}, i.e.\@ a braided groupoid with a single vertex, then the map $x^{-\bullet}\rightharpoonup\blank$ is independent of $x$ if and only if the action $\rightharpoonup$ is trivial; if and only if the group $\Gg$ is abelian and the braiding $\sigma$ is the canonical flip. This is an easy verification.}; thus for each pair $(\zeta,\lambda)$ we need to make a choice of the arrow $x$ that we want to use in this construction. Lemma \ref{lem:maximal_PH_subgroupoid} comes to our rescue: we can choose a maximal subgroupoid of pairs $\Gg'$ (which is wide by Lemma \ref{lem:maximal_PH_subgroupoid}), and for all pairs $(\zeta, \lambda)$ we choose the unique arrow $x\colon \lambda\to \mu$ that lies in $\Gg'$.


\begin{remark}\label{rem:tetrahedron_commutes}Let $\Gg$ be a braided groupoid, and select a maximal subgroupoid of pairs $\Gg'$. Let $\varphi$ be defined by the above construction, after a choice of $\zeta\in\lambda$ and $\varphi_\zeta\colon \Gg(\zeta,\lambda)\to A$, and after the choice of $\Gg'$. Given arrows $x\colon \lambda\to \mu$ and $y\colon \mu\to \nu$ in $\Gg'$, the following diagram commutes:
\[\begin{tikzcd}& \Gg(\mu,\Lambda)\ar[ddd,"\varphi_\mu"]\ar[rd,"y^{-\bullet}\rightharpoonup\blank"]&\\
\Gg(\lambda,\Lambda)\ar[ru,"x^{-\bullet}\rightharpoonup\blank"]\ar[rr,crossing over,"(x\bullet y)^{-\bullet}\rightharpoonup\blank"{fill=white}]\ar[rdd, "\varphi_\lambda"{below left}]& & \Gg(\nu,\Lambda)\ar[ldd, "\varphi_\nu"{below right}]\\
&&\\
&A&
\end{tikzcd}\]
Indeed, the diagram commutes if $(x\bullet y)^{-\bullet}\rightharpoonup u =y^{-\bullet}\rightharpoonup (x^{-\bullet}\rightharpoonup u)$ holds for all $u\in \Gg(\lambda,\Lambda)$, but this is immediate because $\rightharpoonup$ is a left action. Notice that, in the above diagram, $x\bullet y$ also lies in $\Gg'$, because $\Gg'$ is a left subgroupoid.\end{remark}
\begin{theorem}\label{thm:aredbraces}
    The labelling $\varphi$ constructed above is a Matsumoto--Shimizu labelling on $\Gg$. Therefore, every connected braided groupoid is obtained from a zero-symmetric dynamical skew brace.
\end{theorem}
\begin{proof} We fix the vertex $\zeta$ and the bijection $\varphi_\zeta$. If $x$ is the unique arrow in $\Gg'(\zeta,\lambda)$, the construction forces $\varphi_\lambda(u) := \varphi_\zeta(x\rightharpoonup u)$, and $\varphi_\lambda^{-1}(a) = x^{-\bullet}\rightharpoonup \varphi_\zeta^{-1}(a)$. One has
\begin{align*}
    a\,\tilde{\cdot}_\lambda \, b &= \varphi_\lambda(\varphi_\lambda^{-1}a\cdot_\lambda \varphi_\lambda^{-1}(b))\\
    &= \varphi_\zeta\bigg( x\rightharpoonup\Big( (x^{-\bullet}\rightharpoonup \varphi_\zeta^{-1}(a))\cdot_\lambda (x^{-\bullet}\rightharpoonup\varphi_\zeta^{-1}(b))   \Big) \bigg)\\
    \overset{(\dagger)}&{=}\varphi_\zeta\bigg( \Big(x\rightharpoonup (x^{-\bullet}\rightharpoonup \varphi_\zeta^{-1}(a))\Big) \cdot_\zeta \Big(x\rightharpoonup (x^{-\bullet}\rightharpoonup\varphi_\zeta^{-1}(b))   \Big) \bigg)\\
    &= \varphi_\zeta(\varphi_\zeta^{-1}(a)\cdot_\zeta\varphi_\zeta^{-1}(b))\\
    &= a\,\tilde{\cdot}_\zeta\, b,
\end{align*}
where the step marked with $(\dagger)$ follows from Proposition \ref{skew-brace-equivalent-conditions} (\textit{iii}).
\end{proof}
Every (braided) groupoid is a union of (braided) connected groupoids, whence the following:
\begin{corollary}
    Every quiver-theoretic skew brace is isomorphic to a disconnected union of dynamical skew braces. 
    
    In particular, this defines an equivalence between the subcategory $\mathsf{connQTSB}^0_\Lambda\subset \mathsf{QTSB}_\Lambda$ of connected skew bracoids, and the subcategory $\mathsf{connDSB}^0_\Lambda\subset \mathsf{DSB}_\Lambda$ of connected zero-symmetric dynamical skew braces. 
    
    More generally, for every weak morphism $f\colon \Gg\to \Hh$ of skew bracoids, there exist suitable Matsumoto--Shimizu labellings on $\Gg$ and $\Hh$, yielding dynamical skew braces $A$ and $B$ respectively, such that $f$ induces a weak morphism between the two canonical skew bracoids associated with $A$ and $B$ respectively.
\end{corollary}
\begin{proof}
    The construction that interprets a dynamical skew brace as a quiver-theoretic skew brace is clearly functorial. Conversely, let $f= (f^1, f^0)\colon \Gg\to \Hh$ be a morphism of connected skew bracoids, over $\Lambda$ and $\Mu$ respectively. We prove that the construction of the Matsumoto--Shimizu labellings can be done in a way that commutes with the map $f$.
    
    Choose a vertex $\zeta \in \Lambda$, and a bijection $\varphi_\zeta\colon \Gg(\zeta,\Lambda)\to A$. Let $\xi := f^0(\zeta)\in \Mu$. We divide the set $\Hh(\xi,\Mu)$ in two pieces: the image $f^1(\Gg(\zeta,\Lambda))$ of $\Gg(\zeta,\Lambda)$, and the remainder $Z:= \Hh(\xi,\Mu)\smallsetminus\im(f^1)$. Define $B:=  (A/\sim) \sqcup Z $, where $\sim$ is the equivalence relation 
    \[ a\sim b \iff f^1(\varphi_\zeta(a)) = f^1(\varphi_\zeta(b)),\]
    let $[a]$ be the class of $a$ modulo $\sim$, and define \[\psi_{\xi}\colon \Hh(\xi,\Mu)\to B,\quad   \psi_{\xi}(x):= \begin{cases}
       [\varphi_\zeta((f^1)^{-1}(x))]&\text{ if }x\in \im f^1 \\
       x&\text{ if } x\in Z.
    \end{cases}\]
    If $g\colon A\to B$ is the map with image $A/\sim$ that sends $a$ to $[a]$, it is clear by construction that the following square commutes:
    \[\begin{tikzcd}
        \Gg(\zeta,\Lambda)\ar[r,"\varphi_\zeta"]\ar[d,"f^1"]& A\ar[d, "g"]\\
        \Hh(\xi,\Mu)\ar[r,"\psi_{\xi}"]& B
    \end{tikzcd}\]
    Choose, as above, a maximal subgroupoid of pairs $\Gg'$ of $\Gg$. Clearly, $f^1(\Gg')$ is a subgroupoid of pairs of $\Hh$, and thus can be completed to a maximal subgroupoid of pairs $\Hh'$. we can thereby perform the constructions of the two Matsumoto--Shimizu labellings, in $\Gg$ with respect to $\zeta$, $A$, $\varphi_\zeta$, and $\Gg'$; and in $\Hh$ with respect to $\xi$, $B$, $\psi_\xi$, and $\Hh'$. It is clear that $f$ intertwines the resulting dynamical skew brace structures. If $f$ is moreover a morphism over $\Lambda$, this property is preserved, and hence $f$ induces a morphism in $\mathsf{DSB}_\Lambda$. It is clear that this construction preserves identity morphisms, and compositions. Therefore, $\mathsf{connQTSB}^0_\Lambda$ and $\mathsf{connDSB}^0_\Lambda$ are equivalent. Notice that different choices of the initial data, in the construction of a Matsumoto--Shimizu labelling, yield generally distinct (albeit isomorphic) dynamical skew braces, thus we cannot expect this equivalence of categories to be an isomorphism.
\end{proof}
\subsection{The case of groupoids of pairs} The construction of \S\ref{sec:constructing} depends heavily on the choice of the maximal subgroupoid of pairs $\Gg '$. However, when $\Gg$ is already a groupoid of pairs, one has $\Gg ' = \Gg$ and the ambiguity is removed. In this case, $A$ is in bijection with $\Lambda$, hence we can just assume $A = \Lambda$. As usual, we denote by $[a,b]$ the unique arrow $a\to b$, and by $[a\Vert b]$ the unique arrow in $\Gg(a,\Lambda)$ labelled by $b$.
\begin{remark}\label{rem:construction_for_PHG}
The construction of \S\ref{sec:constructing}, in case $\Gg$ is a groupoid of pairs, depends only on two data: the vertex $\zeta$, from which the construction begins; and the bijection $\varphi_\zeta$. If we choose a different bijection $\varphi_\zeta'$, the two resulting labellings only differ by a permutation of the set of labels, so this latter choice is not very important. In particular, we can always choose $\varphi_\zeta$ so that the neutral element for $\cdot_\zeta$ is $\zeta$.

Although there are many ways to endow $\Gg$ with a Matsumoto--Shimizu labelling, the following is a standard way, that always works. Choose a vertex $\zeta$, and define $\varphi_\zeta([\zeta, a]):= a$ for all $a\in \Lambda$; that is, $[\zeta\Vert a]:= [\zeta, a]$. Since $[\zeta\Vert \zeta] = [\zeta,\zeta]$ is the unit of $\cdot_\zeta$, the group operation $\cdot := \tilde{\cdot}_\zeta$ on $\Lambda$ has unit $\zeta$. For a general arrow $[a,b]$ in $\Gg$, we define $\varphi_a([a,b]):= a^{-1}b$; that is, $[a\Vert a^{-1}b] = [a, b]$. 

We check that this is indeed a Matsumoto--Shimizu labelling, and that the associated family $\set{\bullet_\lambda}_{\lambda\in\Lambda}$ represents indeed the groupoid structure given in Remark \ref{rem:indeed_are_groupoids}. It is clear that $\varphi_a([a,a]) = a^{-1}a = \zeta$. Observe that $\varphi^{-1}_a(\beta) = \target(\beta)$ does not depend on $a$. Moreover, one has $a\bullet_\lambda (a^{-1}b) = aa^{-1}b = b$, whence $\target([a\Vert a^{-1}b]) = b$, and \[  [a,b]\bullet [b,c]  = [a\Vert a^{-1}b]\bullet [b\Vert b^{-1}c] = [a\Vert a^{-1}c], \] thus $[a\Vert a^{-1}b]$ represents indeed $[a,b]$, and the associated family $\set{\bullet_\lambda}_{\lambda\in\Lambda}$ represents the groupoid multiplication $\bullet$.
%
\end{remark}
The special Matsumoto--Shimizu labelling defined in Remark \ref{rem:construction_for_PHG} induces a group operation $\cdot$ on $\Lambda$. On the other hand, the braiding $\sigma$ on $\Gg$ is identified by the ternary operation $\langle \blank,\blank,\blank\rangle$ on $\Lambda$ defined in \eqref{eq:sigma_ternary}, and this by Proposition \ref{prop:braiding-iff-ternary-iff-group} induces a group structure $\cdot'$ on $\Lambda$,
by setting \[  a\cdot' b := \varphi \langle \varphi^{-1}(a),\varphi^{-1}(0), \varphi^{-1}(b)\rangle.\] 
\begin{proposition}\label{prop:cdot_prime_equals_cdot}
    One has $\cdot' = \cdot$.
\end{proposition}
\begin{proof} Recall that the ternary operation defined in \eqref{eq:sigma_ternary} describes the ``middle vertex'' of $\sigma([a,b]\ot [b,c])$ as a function of $a,b,c$. The ``middle vertex'' of $\sigma([\phi(\zeta, a), \zeta]\ot [\zeta, \phi(\zeta, b)])$ is the target of $[\phi(\zeta,a), \zeta]\rightharpoonup [\zeta, \phi(\zeta, b)]$. Observe that
    \begin{align*}
        [\phi(\zeta,a), \zeta]\rightharpoonup [\zeta, \phi(\zeta, b)]&= [\zeta\Vert a]^{-\bullet}\rightharpoonup [\zeta\Vert b]\\
        \overset{(\ddagger)}&{=} [\zeta\Vert a]^{-\bullet}\bullet [\zeta\Vert a\cdot b],
    \end{align*}
    where the equality marked with $(\ddagger)$ follows from Remark \ref{rem:bullet-as-cdot}. Thus, the target of $[\zeta\Vert a]^{-\bullet}\bullet [\zeta\Vert a\cdot b]$ is $\phi(\zeta, a\cdot b)$, whence $\langle \phi(\zeta, a),\zeta, \phi(\zeta, b)\rangle = \phi(\zeta, a\cdot b)$. One thereby has
    \begin{align*}
         a\cdot' b &= \varphi \langle \varphi^{-1}(a),\varphi^{-1}(0), \varphi^{-1}(b)\rangle\\
         &= \varphi \langle \phi(\zeta, a),\zeta, \phi(\zeta, b)\rangle\\
         &= \varphi(\phi(\zeta, a\cdot b))\\
         &= a\cdot b
    \end{align*}
    as desired.\end{proof}
\begin{example}
    As in Example \ref{ex:Z4}, we consider the subquiver $\Kk_{4,5,6,7}$ of the maximal zero-symmetric dynamical skew brace of $A = \Z/4\Z$. This is a braided groupoid of pairs, and hence corresponds to a ternary operation on $\Obj(\Kk_{4,5,6,7})$. For brevity, we set $a:=S_4, b:= S_5, c := S_6, d:= S_7$. The ternary operation $\langle\blank,\blank,\blank\rangle$ is given by:
\begin{align*}
     &\langle a,b,a\rangle =d &&\langle a,d,a\rangle =   b  & &\langle a,d,b\rangle = c&&\langle a,c,b\rangle =  d& &\langle a,c,d\rangle =b &&\langle a,b,d\rangle = c\\
     &\langle b,a,b\rangle =c &&\langle b,c,b\rangle =a  &&\langle b,c,a\rangle = d&&\langle b,d,a\rangle =  c&&\langle b,d,c\rangle =a &&\langle b, a, c\rangle = d\\
     & \langle c,d,c\rangle = b&&\langle c,b,c\rangle = d&&\langle c,d,b\rangle =a &&\langle c,a,b\rangle = d && \langle c,b,d\rangle = a&&\langle c,a,d\rangle =b  \\
     &\langle d,c,d\rangle = a &&\langle d,a,d\rangle = c &&\langle d,b,a\rangle =c &&\langle d,c,a\rangle = b &&\langle d,b,c\rangle =a &&\langle d,a,c\rangle = b
 \end{align*}
plus the \emph{Mal'tsev conditions}:
$$\langle i,i,j\rangle = j,\quad \langle i,j,j\rangle = i \quad\text{for all }i,j.$$
We search for an abelian group operation $\cdot$ on $\Obj(\Kk_{4,5,6,7})$ such that the ternary operation has the form $\langle a,b,c\rangle = a b^{-1} c $. The group $(\Obj(\Kk_{4,5,6,7}),\cdot)$ must be isomorphic to $\Z/4\Z$, as it is easily verified. The following is an exhaustive list of suitable group structures, where each group structure is defined by imposing an isomorphism $(\Obj(\Kk_{4,5,6,7}),\cdot)\to (\Z/4\Z,+)$:
\begin{align*}
&a \mapsto 0,\quad b\mapsto 1,\quad c\mapsto 2,\quad d\mapsto 3; &&
&a \mapsto 0,\quad b\mapsto 3,\quad c\mapsto 2,\quad d\mapsto 1;\\
&b \mapsto 0,\quad a\mapsto 1, \quad c\mapsto 3,\quad d\mapsto 2; &&
&b \mapsto 0,\quad a\mapsto 3, \quad c\mapsto 1,\quad d\mapsto 2;\\
&c \mapsto 0,\quad a\mapsto 2,\quad b\mapsto 1, \quad d\mapsto 3;&&
&c \mapsto 0,\quad a\mapsto 2,\quad b\mapsto 3, \quad d\mapsto 1;\\
&d \mapsto 0,\quad a\mapsto 1, \quad b\mapsto 2,\quad c\mapsto 3; &&
&d \mapsto 0,\quad a\mapsto 3, \quad b\mapsto 2,\quad c\mapsto 1.\\
\end{align*}
Among these, the labellings $\varphi$ that satisfy the requirements of Proposition \ref{prop:cdot_prime_equals_cdot} are only one for each choice of the vertex $\zeta$ mapping to $0$:
\begin{align*}
&a \mapsto 0,\quad b\mapsto 3,\quad c\mapsto 2,\quad d\mapsto 1;\\
&b \mapsto 0,\quad a\mapsto 1, \quad c\mapsto 3,\quad d\mapsto 2;\\
&c \mapsto 0,\quad a\mapsto 2,\quad b\mapsto 1, \quad d\mapsto 3;\\
&d \mapsto 0,\quad a\mapsto 1, \quad b\mapsto 2,\quad c\mapsto 3.\\
\end{align*}
\end{example}
\bigskip

\noindent\textbf{Acknowledgements.} The author is grateful to Alessandro Ardizzoni for his advice, for his careful reading of some sections of this work, and for pointing out Remark \ref{rem:slice} and other insightful remarks; to Youichi Shibukawa for his remarks and suggestions during the early phase of this research, and for his collaboration in proving Remark \ref{rem:a-iff-a1a2} and Proposition \ref{prop:braiding-iff-ternary-iff-group}; and to an anonymous Referee for suggesting \S\ref{sec:post} and other significant improvements to this paper. This research was supported by the University of Turin through a PNRR DM 118 scholarship; by the project
OZR3762 of Vrije Universiteit Brussel; through the
\textsc{FWO} Senior Research Project G004124N; and by the European Union -- NextGenerationEU under \textsc{{NRRP}}, Mission 4 Component 2 CUP D53D23005960006 -- Call \textsc{{PRIN}} 2022 No.\@ 104 of February 2, 2022 of Italian Ministry of University and Research; Project 2022S97PMY \textit{Structures for Quivers, Algebras and Representations} (\textit{{SQUARE}}).
\begin{CJK}{UTF8}{gbsn}
\bibliographystyle{abbrv}
\bibliography{refs}
\end{CJK}
\end{document}